\documentclass[12pt]{amsart}
\usepackage[all]{xy}
\usepackage{amssymb}
\usepackage{amsthm}
\usepackage{amsmath}
\usepackage{amscd,enumitem}
\usepackage{verbatim}
\usepackage{eurosym}
\usepackage{graphicx}
\usepackage{dsfont}
\usepackage{stmaryrd}
\usepackage{color}
\usepackage{float}
\usepackage{longtable}
\usepackage{dcolumn}
\usepackage[mathscr]{eucal}
\usepackage[all]{xy}
\usepackage{hyperref}
\usepackage[title]{appendix}
\usepackage[usenames,dvipsnames]{xcolor}
\usepackage{bbm}
\usepackage[textheight=8.85in, textwidth=6.8in]{geometry}
\usepackage{multirow}
\usepackage{caption}
\theoremstyle{plain}
\newtheorem{theorem}{Theorem}[section]
\newtheorem{corollary}[theorem]{Corollary}
\newtheorem{lemma}[theorem]{Lemma}
\newtheorem{proposition}[theorem]{Proposition}

\theoremstyle{definition}
\newtheorem{definition}[theorem]{Definition}

\theoremstyle{remark}
\newtheorem*{remark}{Remark}

\newcommand{\Z}{\mathbb{Z}}
\newcommand{\Q}{\mathbb{Q}}

\newcommand{\ch}{\mathrm{char}}
\newcommand{\F}{\mathbb{F}}

\newcommand{\im}{\mathrm{Im}}

\global\long\def\Gal{\operatorname{Gal}}

\newcommand{\R}{\mathbb{R}}
\newcommand{\HH}{\mathbb{H}}

\newcommand{\N}{\mathbb{N}}
\newcommand{\SL}{\operatorname{SL}}

\newcommand{\Leg}{\mathrm{Leg}}
\newcommand{\C}{\mathbb{C}}

\newcommand{\leg}[2]{\genfrac{(}{)}{}{}{#1}{#2}}

\catcode`,\active

\catcode`\,12

\numberwithin{equation}{section}

\begin{document}
\title[Kloosterman sums and Hurwitz class number]
{Fourth power moment of twisted Kloosterman sum and Hurwitz class numbers}

%    author one information
 %\author[short version for running head]
{}
\author{Neelam Saikia}
\address{Department of Mathematics, Indian Institute of Technology Bhubaneswar}
\curraddr{}
\email{neelamsaikia@iitbbs.ac.in}
\thanks{}

%    author two information

%    \subjclass is required.
%\date{April 25, 2019}
\thanks{}
\begin{abstract} 
In this paper, we investigate the fourth power moment of twisted Kloosterman sum and its relationship with Hurwitz class number. We derive an explicit formula expressing this moment in terms of weighted sums involving Hurwitz class numbers. Our approach involves analyzing point counting formulas associated with the resolution of certain Calabi–Yau threefold. Furthermore, we study the asymptotic behaviour of weighted sums of Hurwitz class numbers that appear in the moment formula.
To derive these asymptotic formulas, we employ the theory of harmonic Maass forms, mock modular forms and holomorphic projections. As an application of these asymptotic results, we obtain the asymptotic formula for the fourth power moment of twisted Kloosterman sums. 
\end{abstract}
\maketitle
\section{Introduction and statement of results}
Let $p$ be an odd prime and $\F_p$ denote a finite field with $p$ elements. Let $\theta:\mathbb{F}_p\rightarrow\mathbb{C}$ be the additive character defined by $\theta(x):= e^{\frac{2\pi ix}{p}}$ and $\overline{x}$ denote the multiplicative inverse of $x\in\mathbb{F}_p\setminus\{0\}$. For $a\in\mathbb{F}_p$, the classical Kloosterman sum is defined by
\begin{align}\label{eqn-1}
K(a,p):=\sum_{x\in\mathbb{F}_p^{\times}}\theta(x+a\overline{x}).
\end{align}

The first appearance of such exponential sums can be found in the works of Poincar\'{e} \cite{Poincare} and Kloosterman \cite{Kloosterman}. Studying estimates on Kloosterman sums play a central role in modern analytic number theory. For instance, Kloosterman established that for any odd prime $p$
$$
|K(a,p)|\leq 3^{1/4}p^{3/4}.
$$
However, the best known current bound is due to Weil \cite{Weil},
$$
|K(a,p)|\leq 2p^{1/2}.
$$
Among with other interesting problems, one of the intriguing problem is to investigate the power moments of Kloosterman sums, namely
$$S(n)_p:=\sum\limits_{a=0}^{p-1}K(a,p)^n.$$

By using elementary methods one can easily obtain

\begin{align}
    S(1)_p&=1,\ S(2)_p=p^2-p-1,\notag\\
    S(3)_p&=\leg{p}{3}p^2+2p=1, \ S(4)_p=2p^3-3p^2-1,\notag
\end{align}
where $\leg{\cdot}{p}$ is the Legendre symbol. Let $\theta_{a,p}\in[0,\pi]$ be such that $p^{-1/2}K(a,p)=2\cos(\theta_{a, p})$ which is called Kloosterman sum angles. Then the vertical Sato-Tate law is guaranteed for these angles that is the limiting distribution of the set $\{\theta_{a,p}: 0\leq a\leq p-1\}$ is semicircular as $p\rightarrow\infty.$ For reference, see the remarkable works of Katz \cite{Katz-1} and Adolphson \cite{Adolphson}. Kloosterman sums have wide impact on the theory of automorphic forms and related areas, see for example \cite{Iwanic}.
Further works on Kloosterman sums can be found in \cite{Schmidt, PTM, ce, di, ronald, Livne, Sa, Da, Hu, ronald, Moisio} etc. Moreover, introducing a character twist makes these exponential sums even more interesting.
To this end we recall power moments of twisted Kloosterman sums. Let $\chi$ be a multiplicative character of $\F_p$ (with the convention that $\chi(0)=0).$ Then the $n$-th moment of twisted Kloosterman sum is defined by
$$
S(n,\chi)_p=\sum\limits_{a=0}^{p-1}\chi(a)K(a,p)^n.
$$

%and with the help of this result along with Chebyshev polynomials Xi and Yi \cite{} computed asymptotic formulas for any even power moments of Kloosterman sums along with a bound for the odd power moments in the growth of $p.$ 
%%%%%%%%%%%%%%%%%%%%%%%%%%%%%%%

If $n=1,$ then it is easy to verify that $S(1,\chi)_p=g(\chi)^2=O(p).$ For $n=2,$ Liu \cite{Liu} expressed this sum in terms of Jacobi sum. For further study on twisted Kloosterman sums we refer to \cite{Conrey, Liu, lin-tu} etc.
In this paper, our object of interest is the fourth power moment of twisted Kloosterman sum, namely $S(4,\phi)_p,$ where $\phi$ is a quadratic character of $\F_p.$ We express this moment as weighted sums of Hurwitz class numbers which enumerate certain isomorphism classes of elliptic curves.
To state this result precisely, we recall some basic notation and definitions.
	If $-D<0$ such that $-D\equiv0,1\pmod{4},$ then $\mathcal{O}(-D)$ denotes the unique imaginary quadratic order with discriminant $-D.$ Let $h(D)=h(\mathcal{O}(-D))$ denote\footnote{We note that $H(D)=H^*(D)=h(D)=0$ whenever $-D$ is neither zero nor a negative discriminant.} the order of the class group of $\mathcal{O}(-D)$ and let $\Omega(D)=\Omega(\mathcal{O}(-D))$ denote half the number of roots of unity in $\mathcal{O}(-D).$  In this notation, we define the Hurwitz class numbers
	\begin{equation}
		H(D):=\sum\limits_{\mathcal{O}\subseteq\mathcal{O'}\subseteq\mathcal{O}_{\text{max}}}h(\mathcal{O'})
		\ \ \ {\text {\rm and}}\ \ \ 
		\ \ H^{\ast}(D):=\sum\limits_{\mathcal{O}\subseteq\mathcal{O'}\subseteq\mathcal{O}_{\text{max}}}\frac{h(\mathcal{O'})}{\Omega(\mathcal{O'})},
	\end{equation}
	where the sum is over all orders $\mathcal{O'}$ between $\mathcal{O}$ and the maximal order $\mathcal{O}_{\text{max}}.$ Here note that we have $H(D)=H^{\ast}(D)$ unless $\mathcal{O}_{\text{max}}=\Z[i]$ or $\Z\left[\frac{-1+\sqrt{-3}}{2}\right],$ where the terms corresponding to $\mathcal{O}_{\text{max}}$ differs by a factor of $2$ and $3$ respectively. As our first result we obtain the following:
	
\begin{theorem} \label{relation-2}
Suppose that $p>3$ be a prime and $-2\sqrt{p}<s<2\sqrt{p}$ be an integer. Then the following are true:\\

\noindent (1) If $p\equiv1\pmod{4},$ then 

$$
S(4,\phi)_p=-p^3+2p^2+4p\sum\limits_{s\equiv p+1\pmod{8}} H^*\left(\frac{4p-s^2}{4}\right)s^2 + 8p\sum\limits_{s\equiv p+1\pmod{16}} H^*\left(\frac{4p-s^2}{16}\right) s^2.
$$

\noindent (2) If $p\equiv3\pmod{4},$ then
$$
S(4,\phi)_p=-p^3+2p^2+4p\sum\limits_{s\equiv p+1\pmod{8}} H^*\left(\frac{4p-s^2}{4}\right) s^2.
$$

\end{theorem}

In the next two theorems we obtain asymptotic formulas of these weighted sums of Hurwitz class number. To estimate this moment, we employ the theory of harmonic Maass forms, mock modular forms and holomorphic projections. 

\begin{theorem}\label{case-1}
As the primes $p\rightarrow\infty$ with $p\equiv1\pmod{4},$ we have 
\begin{align}
\sum\limits_{s\equiv p+1\pmod{8}} H^*\left(\frac{4p-s^2}{4}\right)s^2=\frac{p^2}{6}+O_{\epsilon}(p^{\frac{3}{2}+\epsilon})\notag
\end{align}
for all $\epsilon>0.$

\end{theorem}

\begin{theorem}\label{case-3}
As the primes $p\rightarrow\infty$ with $p\equiv1\pmod{4},$ we have
\begin{align}
12\sum\limits_{s\equiv p+1\pmod{16}} H^*\left(\frac{4p-s^2}{16}\right)s^2=\frac{p^2}{2}+O_{\epsilon}(p^{\frac{3}{2}+\epsilon})\notag
\end{align}
for all $\epsilon>0.$
\end{theorem}

\begin{theorem}\label{case-2}
As the primes $p\rightarrow\infty$ with $p\equiv3\pmod{4},$ we have 
\begin{align}
\sum\limits_{s\equiv p+1\pmod{8}}H^*\left(\frac{4p-s^2}{4}\right)s^2=\frac{p^2}{4}+O_{\epsilon}(p^{\frac{3}{2}+\epsilon})\notag
\end{align}
for all $\epsilon>0.$
\end{theorem}

As a corollary of our results, we obtain the asymptotic formula for $S(4,\phi)_p.$

\begin{corollary}\label{twisted-moment-1}
As the primes $p\rightarrow\infty,$ we have 

$$
S(4,\phi)_p=O_{\epsilon}(p^{\frac{5}{2}+\epsilon})
$$
for all $\epsilon>0.$

\end{corollary}

\noindent For $a\neq0$, let $\pi(a)$ and $\overline{\pi}(a)$ denote the roots of the polynomial $X^2+K(a,p)X+p.$ Then one can write down the Kloosterman sum $K(a,p)$ in terms of the roots $\pi(a)$ and $\overline{\pi}(a).$ Now, consider the following exponential sum twisted by a character, namely
\begin{align}
M(n,\phi)_p:=\sum_{a\in\mathbb{F}_p^{\times}}\phi(a)(\pi(a)^n+\pi(a)^{n-1}\overline{\pi}(a)+\pi(a)^{n-2}\overline{\pi}(a)^2+\cdots +\pi(a)\overline{\pi}(a)^{n-1}+\overline{\pi}(a)^n).\notag
\end{align}
This is known as the $n$-th twisted Kloosterman sheaf sum of $\phi.$ Evans \cite{evans3} predicted relations of twisted Kloosterman sheaf sums with Gaussian hypergeometric functions and conjectured several relations connecting these twisted Kloosterman sheaf sums to the Fourier coefficients of modular forms. One of these conjectures was settled in \cite{DGP}.  As a consequence of 
Corollary \ref{twisted-moment-1}, we obtain asymptotic formula for $M(4,\phi)_p$ as $p$ grows.

\begin{corollary}\label{twisted-moment-2}
    \label{moment-1}
As the primes $p\rightarrow\infty,$ we have 

$$
M(4,\phi)_p=O_{\epsilon}(p^{\frac{5}{2}+\epsilon})
$$
for all $\epsilon>0.$

\end{corollary}

\begin{remark}
It is important to note that Corollary \ref{twisted-moment-1} and  Corollary \ref{twisted-moment-2} can also be derived independently by making use of Evan's conjecture \cite[Theorem 1.1]{DGP} and Theorem \ref{Deligne}. Therefore, Theorem \ref{case-2} can also be established by applying \cite[Theorem 1.1]{DGP} and Theorem \ref{Deligne} independently. However, Theorems \ref{case-1} and \ref{case-3} will not be settled in the same manner.
Indeed our approach to prove these theorems is rely on the theory of harmonic Maass forms, mock modular forms and holomorphic projections. 
\end{remark}

%%%%%%%%%%%%%%%%%%%%%%%%%%%
%%%%%%%%%%%%%%%%%%%%%%%%%%%%

The rest of the paper is structured as follows. In Section 2, we discuss Legendre elliptic curves and their quadratic twists. Section 3 is devoted to deriving the fourth power moment of twisted Kloosterman sums as second moment of  traces of Frobenius of certain family of Legendre elliptic curves. In Section 4, we prove Theorem \ref{relation-2}.  In Section 5, we briefly review harmonic Maass forms, Rankin–Cohen bracket operators, holomorphic projections, and certain results of Mertens. Using these tools, we establish asymptotic formulas for the weighted class number sums. In Section 6, we complete the proof of Theorems \ref{case-1}--\ref{case-2}. 
While deriving our results, we encounter certain intermediate consequences, which, in turn, lead to asymptotic average values of specific families of hypergeometric functions. We discuss these results in Section 7.

\section{Arithmetic of elliptic curves} 
In this section, we discuss several important facts about Legendre elliptic curves. These results will be used in the derivation of the fourth power moment of twisted Kloosterman sums. For $\lambda\in\F_p\setminus\{0,1\},$ let $$E^{\Leg}_{\lambda}:\ \ y^2=x(x-1)(x-\lambda)$$ be the Legendre elliptic curve over $\F_p,$ where $p>3$ be a prime and define
$$
a_p(\lambda):=p+1-|E^{\Leg}_\lambda(\F_p)|
$$
to be the trace of Frobenius of $E^{\Leg}_\lambda,$ where $|E^{\Leg}_\lambda(\F_p)|$ denote the number of $\F_p$-rational points of $E^{\Leg}_\lambda.$ It is straightforward to verify that $$a_p(\lambda)=-\sum\limits_{x\in\F_p}\phi(x(x-1)(x-\lambda)),$$ where $\phi$ is the quadratic character of $\F_p.$ Moreover, due to Hasse we have $|a_p(\lambda)|\leq 2\sqrt{p}.$

\begin{proposition}[Proposition 1.7, Chapter III of \cite{Silverman}]\label{elliptic-proposition-1}
Let $K$ be a field with $\ch(K)\neq2,3$.

\noindent
(1) Every elliptic curve $E\slash K$ is isomorphic over $\overline{K}$ to an elliptic curve $E^{\Leg}_\lambda.$

\noindent
(2)  If $\lambda\neq0,1$, then the $j$-invariant of $E^{\Leg}_\lambda$ is
$$j(E^{\Leg}_\lambda)=2^8\cdot\frac{(\lambda^2-\lambda+1)^3}{\lambda^2(\lambda-1)^2}.$$

\noindent
(3) The only $\lambda$ for which $j(E^{\Leg}_\lambda)=1728$ are $\lambda=2,-1,$ and $1/2$.

\noindent 
(4) The only $\lambda$ for which $j(E^{\Leg}_\lambda)=0$ are 
$\lambda=\frac{1\pm\sqrt{-3}}{2}.$

\noindent
(5)  For every $j\not \in \{0, 1728\},$ the map $K\setminus\{0,1\}\to j(E^{\Leg}_\lambda)$ is six to one. In particular, we have
$$
\left\{\lambda, \frac{1}{\lambda}, 1-\lambda, \frac{1}{1-\lambda}, \frac{\lambda}{\lambda-1},
\frac{\lambda-1}{\lambda}\right\}\to j(E^{\Leg}_\lambda).
$$
\end{proposition}

The following proposition provides the classification of quadratic twists of Legendre elliptic curves:

\begin{proposition}[Proposition 3.2 of \cite{ahl-ono-1}]\label{twists}
Suppose that $p\geq5$ is a prime and $\lambda\in\F_p\backslash\{0,1\}.$

\noindent
(1) $E^{\Leg}_{\lambda}$ is the $\lambda$ quadratic twist of $E^{\Leg}_{1/\lambda}.$

\noindent
(2) $E^{\Leg}_{\lambda}$ is the $-1$ quadratic twist of $E^{\Leg}_{1-\lambda}.$

\noindent
(3) $E^{\Leg}_{\lambda}$ is the $1-\lambda$ quadratic twist of $E^{\Leg}_{\lambda/(\lambda-1)}.$
\end{proposition}

We will also need the following two theorems characterizing an elliptic curve $E/\F_p$ whose group of $\F_p$-rational points contain $\Z/2\Z\times \Z/4\Z$ or $\Z/4\Z\times \Z/4\Z.$

\begin{proposition}[Proposition 3.3 (1) of \cite{ahl-ono-1}]\label{isomorphism-1}
Let $E/\F_p$ be an elliptic curve such that its group of $\F_p$-rational points contains the subgroup $\Z/2\Z\times \Z/4\Z.$ Then there exist $\lambda\in\F_p\backslash\{0,\pm1\}$ such that $E^{\Leg}_{\lambda^2}$ is isomorphic over $\F_p$ to the curve $E.$ Furthermore, if $\lambda\in\F_p\backslash\{0,\pm1\},$ then the group of $\F_p$-rational points of $E^{\Leg}_{\lambda^2}$ contains the subgroup $\Z/2\Z\times \Z/4\Z.$
\end{proposition}

\begin{proposition}[Proposition 3.3 (2) of \cite{ahl-ono-1}]\label{isomorphism-2}
Suppose that $\lambda\in\F_p\backslash\{0,\pm1\}.$ Then the group of $\F_p$-rational points of $E^{\Leg}_{\lambda^2}$ contains the subgroup  $\Z/4\Z\times \Z/4\Z$ if and only if $p\equiv1\pmod{4}$ and $\lambda^2-1$ is a square in $\F_p$.
\end{proposition}

To this end, for $\lambda\in\F_p\setminus\{0,\pm1\}$ we define 
$L(\lambda):=\{\pm\mu\in\F_p\setminus\{0,\pm1\}: E_{\lambda^2}^{\Leg}
\cong_{\F_p}E_{\mu^2}^{\Leg}\}.$ The following lemmas provides the size of $L(\lambda)$ depending certain conditions.

\begin{lemma}\label{final-count-1}
 Let $p\equiv1\pmod{4},$ $\lambda\in\F_p\setminus\{0,\pm1\}$ and $j(E_{\lambda^2}^{\Leg})\neq0,1728.$ Then we have the following
 $$|L(\lambda)|=\begin{cases}
     4, & \text{if}\ 1-\lambda^2\neq\square\\
     12,  & \text{if}\ 1-\lambda^2=\square.
 \end{cases}$$
\end{lemma}

\begin{proof}
We proceed by applying Proposition \ref{twists}, and we consider two separate cases based on whether $1-\lambda^2$
  is a square in the underlying field $\F_p.$
\end{proof}
In a similar fashion, the next lemma follows readily from Proposition \ref{twists} by applying the same case-based analysis.
\begin{lemma}\label{final-count-2}
 Let $p\equiv3\pmod{4}$ and $j(E_{\lambda^2}^{\Leg})\neq0,1728.$ Then we have $|L(\lambda)|=4.$
\end{lemma}

For $j=0$ and $j=1728,$ we have the following two lemmas, respectively:

\begin{lemma}\label{final-count-3}
If $E_{\lambda^2}^{\Leg}/\F_p$ has $j(E_{\lambda^2}^{\Leg})=0,$ then the following are true:\\
\noindent (1) If $p\equiv2\pmod{3},$ then there are no $E_{\lambda^2}^{\Leg}$ exists.\\

\noindent (2) If $p\equiv1\pmod{12},$ then $\frac{1\pm\sqrt{-3}}{2}$ both are squares in $\F_p$ and $|L(\lambda)|=4.$\\

\noindent (3) If $p\equiv7\pmod{12},$ then $\frac{1\pm\sqrt{-3}}{2}$ both are not squares in $\F_p$ and $|L(\lambda)|=0.$
\end{lemma}

\begin{proof}
  This follows directly from Lemma 2.6 of \cite{KHN}.  
\end{proof}

\begin{lemma}\label{final-count-4}
If $E_{\lambda^2}^{\Leg}/\F_p$ has $j(E_{\lambda^2}^{\Leg})=1728,$ then the following are true:\\
\noindent (1) If $p\equiv3\pmod{4},$ then $a_p(\lambda)=0.$\\

\noindent (2) If $p\equiv1\pmod{8},$ then $|L(\lambda)|=6.$\\

\noindent (3) If $p\equiv5\pmod{8},$ then $|L(\lambda)|=2.$
\end{lemma}

\begin{proof}
  This follows directly from Lemma 2.5 of \cite{KHN}.  
\end{proof}

	Now, we discuss explicit formulas for counting the number of isomorphism classes of elliptic curves over $\F_q.$  This count is given in terms of Hurwitz class numbers, as established in a theorem due to Schoof \cite{schoof}. 	
	
	\begin{theorem}[Section 4 of \cite{schoof}]\label{Schoof} If $p\geq 5$ is prime, and  $q=p^r,$ then the following are true.

		\noindent
		(1) If $n\geq 2$ and  $s$ is a nonzero integer for which  $p|s$ and $s^2\neq 4q,$  then there are no elliptic curves $E/\F_q$
		with $|E(\F_q)|=q+1-s$ and 
		$\Z/ n\Z \times \Z/ n\Z \subseteq E(\F_q).$
		
		\noindent
		(2)  If $r$ is even and $s=\pm 2p^{r/2},$ then the number of isomorphism classes of elliptic curves over $\F_q$ with
		$\Z/ n\Z \times \Z/ n\Z \subseteq E(\F_q)$ and
		$|E(\F_q)|=q+1-s$ is
		\begin{equation}\label{Ap}
			\frac{1}{12}\left(p+6-4\leg{-3}{p}-3\leg{-4}{p}\right),
		\end{equation}
		where $\leg{\cdot}{p}$ is the Legendre symbol.
		
		\noindent 
		(3) If $r$ is even, $s=\pm p^{r/2}$ and $p\not\equiv1\pmod{3},$ then the number of isomorphism classes of elliptic curves over $\F_q$ with
		$ \Z /n\Z \subseteq E(\F_q)$ and
		$|E(\F_q)|=q+1-s$ is
		\begin{equation}\label{Ap-2}
			1-\leg{-3}{p}.
		\end{equation}
		
		\noindent
		(4)  Suppose that $n$ and $s$ are integers such that $s^2\leq 4q,$ $p\nmid s,$ $n^2\mid (q+1-s),$ and $n\mid (q-1).$ Then the number of isomorphism classes of elliptic curves over $\F_q$ with
		$|E(\F_q)|=q+1-s$ and $\Z/ n\Z \times \Z/ n\Z\subseteq E(\F_q)$ is $H\left(\frac{4q-s^2}{n^2}\right).$

	\end{theorem}

\section{Moments of Twisted Kloosterman sums and class numbers}

\noindent Let $\widehat{\mathbb{F}_p^\times}$ be the group of all multiplicative characters of $\mathbb{F}_p^{\times}$. Let $\overline{\chi}$ denote the inverse of a multiplicative character $\chi$ of $\F_p$ (with the convention that $\chi(0)=0).$ 

\begin{lemma}\emph{(\cite[Chapter 8]{ireland}).}\label{lemma-3} For a prime $p,$ we have
\begin{enumerate}
\item $\displaystyle\sum\limits_{x\in\mathbb{F}_p}\chi(x)=\left\{
                                  \begin{array}{ll}
                                    p-1 & \hbox{if~ $\chi=\varepsilon$;} \\
                                    0 & \hbox{if ~~$\chi\neq\varepsilon$.}
                                  \end{array}
                                \right.$
\item $\displaystyle\sum\limits_{\chi\in \widehat{\mathbb{F}_p^\times}}\chi(x)~~=\left\{
                            \begin{array}{ll}
                              p-1 & \hbox{if~~ $x=1$;} \\
                              0 & \hbox{if ~~$x\neq1$.}
                            \end{array}
                          \right.$
\end{enumerate}
\end{lemma}

\begin{lemma}
    For an odd prime $p>3$ and $a\in\F_p\setminus\{0\}$ we have
    \begin{align}\label{final-lemma-1.1}
     \sum\limits_{y\in\F_p}\phi(y^2+ay) =   \sum\limits_{y\in\F_p}\phi(y^2+a)=-1 
      \end{align}
    
\end{lemma}

 \begin{proof}
 It is routine to verify that $\sum\limits_{y\in\F_p}\phi(y^2+ay) =-1.$ Now, counting the number of solutions of $x^2-y^2\equiv a\pmod{p}$ we aim to show that 
     $$\sum\limits_{y\in\F_p}\phi(y^2+a)=-1.$$     
\noindent To proceed, let $N_p(a)$ denote the number of solution to the congruence $x^2-y^2\equiv a\pmod{p}.$ We can express $N_p(a)$ as $$N_p(a)=\sum\limits_{y\in\F_p}(1+\phi(y^2+a)).$$ On the other hand, by making the substituting $x+y=z$ and $x-y=w$ we can rewtite the congruence as  $$x^2-y^2=(x+y)(x-y)\equiv zw \equiv a\pmod{p}.$$ Since for each nonzero solution $zw \equiv a\pmod{p}$ there are 
$p-1$ such pairs that satisfy that equation, so we conclude that $N_p(a)= p-1.$ Combining both the expressions of $N_p(a)$ we obtain the identity: $\sum\limits_{y\in\F_p}(1+\phi(y^2+a))= p-1.$     
     This completes the proof.
     
      \end{proof}
We now discuss some connecting results that will be crucial in proving the main theorems. The following proposition expresses the sum $S(m,\phi)_p$ as a symmetric sum involving quadratic characters.

\begin{proposition}\cite[Proposition 3.1]{DGP}\label{proposition-1} For an integer $m\geq1,$ we have
\begin{align*}
S(m+1,\phi)_p=p\phi(-1)\sum_{x_1,\ldots,x_m\in\mathbb{F}_p^{\times}}\phi(x_1+x_2+\cdots+x_m+1)\phi(\overline{x_1}+\overline{x_2}+
\cdots+\overline{x_m}+1)
\end{align*}
\end{proposition}

We now prove a proposition that will be the key to relate $S(4,\phi)_p$ and the second moment of the trace of Frobenius of Legendre elliptic curves whose groups contain $\Z/2\Z\times\Z/4\Z.$ To proceed, we first establish a relation between $S(4,\phi)_p$  and the number of $\F_p$-rational points of a resolution of certain Calabi-Yau threefold given by  
\begin{align}\label{new-eq-1}
 X: x+\overline{x}+y+\overline{y}+z+\overline{z}+u+\overline{u}=0.
\end{align} 
It is important to note that van Geemen-Nygaard \cite{Geemen} and Ahlgren-Ono \cite{ahl-ono-1} independently proved that this Calabi-Yau threefold is modular. To obtain the next proposition, we employ a series of transformations that are also used in \cite{ahl-ono-1}. 

\begin{proposition}\label{proposition-hypergeometric}
For an odd prime $p>3$ we have

\begin{align}
S(4,\phi)_p=-p^3+2p^2+p\sum\limits_{\gamma\in\mathbb{F}_p\setminus
\{0,\pm1\}}a_p(\gamma^2)^2.\notag
\end{align}

\end{proposition}

\begin{proof}
Applying Proposition \ref{proposition-1} for $m=3,$ we have

\begin{align}\label{eqn-4}
S(4,\phi)_p&=p\phi(-1)\sum_{x,y,z\in\mathbb{F}_p^{\times}}\phi(x+y+z+1)
\phi(\overline{x}+\overline{y}+\overline{z}+1).
\end{align}

Now, suppose that $C_p:=|\{(x,y,z,u)\in(\F_p^{\times})^4: x+\overline{x}+y+\overline{y}+z+\overline{z}+u+\overline{u}=0\}|.$

In \eqref{new-eq-1}, by taking the transformations $x\rightarrow ux, y\rightarrow uy, z\rightarrow uz$ and $u\rightarrow u,$ we have 
\begin{align}\label{Curve}
u^2(1+x+y+z)=-(1+\overline{x}+\overline{y}+\overline{z}).
\end{align}
 Then the number of $\F_p$-rational points of $X$ can be expressed as

\begin{align}
C_p&=(p-1)\cdot C_p^{(0)}+\sum\limits_{\substack{x,y,z\in\F_p^{\times}\\1+x+y+z\neq0,\ 1+\overline{x}+\overline{y}+\overline{z}\neq0}}(1+\phi(-(1+x+y+z)(1+\overline{x}+\overline{y}+\overline{z}))),
\end{align} 

where $C_p^{(0)}$ is the number of simultaneous solutions of $1+x+y+z=0$ and $1+\overline{x}+\overline{y}+\overline{z}=0.$
Now, simply eliminating $x$ from the second equation by using the first and then by applying standard algebraic manipulations the problem reduces to solving $(y+z)(1+y)(1+z)=0.$ This yields 
$C_p^{(0)}=3(p-2).$ Moreover, note that each of this simultaneous solution $(x,y,z)$ corresponds to $(p-1)$ solutions of \eqref{Curve}, namely $(x,y,z,u),$ where $u\in\F_p\setminus\{0\}.$  
Using this facts in the expression of $C_p,$ we obtain that

\begin{align}\label{final-eq-2}
C_p&=3(p-1)(p-2)+(p-1)^3-2(p-1)^2+3(p-2)\notag\\
&+\phi(-1)\sum\limits_{x,y,z\in\F_p^{\times}}\phi((1+x+y+z)(1+\overline{x}+\overline{y}+\overline{z})).   
\end{align}
Comparing \eqref{eqn-4} and \eqref{final-eq-2} we have

\begin{align}\label{final-eq-3}
C_p=(p-1)^3-2(p-1)^2+3(p-1)(p-2)+3(p-2)+\frac{S(4,\phi)_p}{p}.
\end{align}
On the other hand, replacing $x+1/x$ by $2\alpha,$ $y+1/y$ by $2\beta$ and $z+1/z$ by $2\gamma$ in \eqref{new-eq-1}, if we count $C_p,$ then we deduce that

$$C_p=\sum\limits_{\alpha,\beta,\gamma\in\F_p}(1+\phi(\alpha^2-1))\cdot (1+\phi(\beta^2-1))\cdot(1+\phi(\gamma^2-1))\cdot (1+\phi((\alpha+\beta+\gamma)^2-1)).$$ 

Here note that similar arguments can be found in \cite{ahl-ono-1}.
Now, taking the transformation $\gamma\rightarrow \gamma-\beta$ followed by the transformation $\beta\rightarrow-\beta$ and expanding further by using \eqref{final-lemma-1.1} we have that 

\begin{align}\label{final-eq-5}
C_p=p^3-4p^2+6p-4+\sum\limits_{\gamma\in\F_p}\left(\sum\limits_{\alpha\in\F_p}
\phi(\alpha^2-1)\phi((\alpha+\gamma)^2-1)\right)^{2}=p^3-4p^2+6p-4+A_p,
\end{align} 

where

\begin{align}
A_p&=\sum\limits_{\gamma\in\F_p}\left(\sum\limits_{\alpha\in\F_p}
\phi(\alpha^2-1)\phi((\alpha+\gamma)^2-1)\right)^{2}\notag\\
&=\sum\limits_{\gamma\in\F_p}\left(\sum\limits_{\alpha\in\F_p}
\phi((\alpha-1)(\alpha+\gamma)+1))\phi((\alpha+1)(\alpha+\gamma-1))\right)^{2}.\notag
\end{align} 

Now, multiplying each term under the character and rearranging the terms we may write

$$
A_p=\sum\limits_{\gamma\in\F_p}\left(\sum\limits_{\alpha\in\F_p}
\phi\left((\alpha+\gamma/2)^2-(\gamma/2+1)^2\right)
\phi\left((\alpha+\gamma/2)^2-(\gamma/2-1)^2\right)
\right)^2.
$$

If we make the transformation $\alpha\rightarrow\alpha-\gamma/2,$ and subsequently replacing $\gamma$ by $2(\gamma+1)$ we obtain that

$$
A_p=\sum\limits_{\gamma\in\F_p}\left(\sum\limits_{\alpha\in\F_p}
\phi(\alpha^2-\gamma^2)\phi(\alpha^2-(\gamma+2)^2)
\right)^2.
$$
Furthermore, transforming $\mu\rightarrow(\alpha^2-\gamma^2)$ and then expanding by using \eqref{final-lemma-1.1} we may write

\begin{align}
A_p=p^2-3p+\sum\limits_{\gamma\in\F_p}\left(\sum\limits_{\mu\in\F_p}
\phi(\mu)\phi(\mu-4\gamma-4)\phi(\mu+\gamma^2)
\right)^2-2\sum\limits_{\gamma,\mu\in\F_p}
\phi(\mu)\phi(\mu-4\gamma-4)\phi(\mu+\gamma^2).\notag    
\end{align}
Now, transforming $\mu\rightarrow 4(\gamma+1)\mu$ for $\gamma\neq-1$ and expanding further by using \eqref{final-lemma-1.1} we may write

$$
\sum\limits_{\gamma,\mu\in\F_p}
\phi(\mu)\phi(\mu-4\gamma-4)\phi(\mu+\gamma^2)=1.
$$

Substituting this value in the above expression of $A_p$ and transforming $\mu\rightarrow 4(\gamma+1)\mu$ we obtain

$$
A_p=-3p+p^2+\sum\limits_{\gamma\in\F_p\setminus\{0,-1\}}\left(\sum\limits_{\mu\in\F_p}
\phi(\mu)\phi(\mu-1)\phi\left(\frac{4\gamma+4}{\gamma^2}\mu+1\right)
\right)^2.
$$
Now, making the transformations $\mu\rightarrow1/\mu$ and $\gamma\rightarrow2/(\gamma-1)$ we derive

$$
A_p=-3p+p^2+\sum\limits_{\gamma\in\F_p\setminus\{\pm1\}}\left(\sum\limits_{\mu\in\F_p}
\phi(\mu)\phi(1-\mu)\phi(\mu+\gamma^2-1)
\right)^2.
$$

Finally, taking $\mu\rightarrow1-\mu$ we have 

\begin{align}\label{final-eq-4}
A_p=1-3p+p^2+\sum\limits_{\gamma\in\F_p\setminus\{0,\pm1\}} a_p(\gamma^2)^2.  
\end{align}
Now, substituting \eqref{final-eq-4} into \eqref{final-eq-5} and then substituting the resultant expression into \eqref{final-eq-3} we conclude the proof.

\end{proof}
\section{Proof of Theorem \ref{relation-2}}

In this section we aim to prove Theorem \ref{relation-2}.  We use the results stated in the previous sections. More precisely,  we analyse the structure of $2$-power torsion points of Legendre elliptic curves and apply the work of Deuring and Schoof to obtain a closed formula for $S(4,\phi)_p$ in terms of weighted sums of Hurwitz class numbers. Similar counting arguments can also be found in \cite{ahl-ono-1}.

\begin{proof}[Proof of Theorem \ref{relation-2}]
Using Proposition \ref{proposition-hypergeometric} we may write $S(4,\phi)_p$ as the second moment of $a_p(\lambda^2).$ More precisely, we have 

\begin{align}\label{twisted-sum}
S(p,\phi)_p+p^3-2p^2=p\sum\limits_{\lambda\neq0,\pm1}a_p(\lambda^2)^2.    
\end{align}

Now, Proposition \ref{isomorphism-1} gives that the moments of $a_p(\lambda^2)$ can be interpreted by the moments of the integers $s$ such that $-2\sqrt{p}<s<2\sqrt{p}$ and $s\equiv p+1\pmod{8}.$ We now first prove claim (1) by considering the following three cases:\\

%%%%%%%%%%%%%%%%%%%%%%%%%%%%%%%%%%%%%%%%%%
    
    \noindent Case 1: Suppose that $s\not\equiv p+1\pmod{16}.$ Let $j(E_{\lambda^2}^{\Leg})\neq0,1728.$ Then by Proposition \ref{isomorphism-1} and Theorem \ref{Schoof} we have that the number of isomorphism classes of elliptic curves with $p+1-s$ points is $H^*\left(\frac{4p-s^2}{4}\right)$ and each such isomorphism class contains $E^{\Leg}_{\lambda^2}$ for some $\lambda\neq0,\pm1.$ Moreover, using Lemma \ref{final-count-1} we count that there are $4$ values of $\lambda$ those are present in the same class. Hence, we  have that the contribution of this case to the sum on the right side of \eqref{twisted-sum} is given by
    
$$4p\sum\limits_{s\not\equiv p+1\pmod{16}} H^*\left(\frac{4p-s^2}{4}\right)s^2.$$
    
    \noindent Case 2: Suppose that $s\equiv p+1\pmod{16}$ and $\Z/4\Z\times \Z/4\Z\subseteq E(\F_p).$ Then proceeding similarly as in Case 1 and using Proposition \ref{isomorphism-2} we have that any isomorphism class of elliptic curves with $p+1-s$ points contains $E^{\Leg}_{\lambda^2}$ such that $1-\lambda^2$ is a square. By Theorem \ref{Schoof} and Lemma \ref{final-count-1} we observe that the contribution of this case to the sum on the right side of \eqref{twisted-sum} is given by
    
$$12p\sum\limits_{s\equiv p+1\pmod{16}} H^*\left(\frac{4p-s^2}{16}\right)s^2.$$

 \noindent Case 3: Suppose that $s\equiv p+1\pmod{16}$ and $\Z/4\Z\times \Z/4\Z\not\subseteq E(\F_p).$ Then any isomorphism class of elliptic curves with $p+1-s$ points are those that contains $\Z/2\Z\times \Z/2\Z$ but not $\Z/4\Z\times \Z/4\Z.$ Therefore, By Theorem \ref{Schoof} and Lemma \ref{final-count-1} we have that the contribution of this case to the required sum is given by
$$4p\sum\limits_{s\equiv p+1\pmod{16}}\left(H^*\left(\frac{4p-s^2}{4}\right)- H^*\left(\frac{4p-s^2}{16}\right)\right)s^2.$$ 
Now, combining all the three cases we obtain that  
\begin{align}\label{final-eq-6}
    p\sum\limits_{\substack{\lambda\neq0,\pm1\\ j\neq0,1728}}a_{p}(\lambda^2)^2=4p\sum\limits_{s\equiv p+1\pmod{8}} H^*\left(\frac{4p-s^2}{4}\right)s^2 + 8p\sum\limits_{s\equiv p+1\pmod{16}} H^*\left(\frac{4p-s^2}{16}\right) s^2.
    \end{align}
    Now, the only remaining possibilities that we need to include in the required sum are those for which $j(E_{\lambda^2}^{\Leg})=0,1728.$ To do this, we perform the same three cases as above and using Proposition \ref{isomorphism-1}, Proposition \ref{isomorphism-2}, Lemma \ref{final-count-3}, Lemma \ref{final-count-4} and Theorem \ref{Schoof} together we conclude the first part of the theorem. We prove claim (2) by proceeding similar arguments as discuss above.
\end{proof}

%%%%%%%%%%%%%%%%%%%%
%%%%%%%%%%%%%%%%%%%%%
%%%%%%%%%%%%%%%%%%%%%

%%%%%%%%%%%%%%%%%%%%%%%%%%%%%%%%%%%%%%%%
%%%%%%%%%%%%%%%%%%%%%%%%%%%%%%%%%%%%%%%%

\section{Harmonic Maass forms and Mock modular forms}\label{HarmonicMaassForms}
In this section we first discuss the theory of harmonic Maass forms and apply this to obtain the asymptotic formulas. For a detailed study of harmonic Maass forms we refer \cite{KO}. Let $\Gamma(\alpha;x):=\int_{\alpha}^{\infty}e^{-t}t^{x-1}dt$ be the incomplete Gamma function. We first recall the following celebrated theorem of Zagier.

\begin{theorem}[\cite{Zagier}]\label{ZagierSeries}
		The function
		$$
		\mathcal{H}(\tau)=-\frac{1}{12}+\sum\limits_{n=1}^\infty H^\ast(n)q_\tau^n+\frac{1}{8\pi\sqrt{y}}+\frac{1}{4\sqrt{\pi}}\sum\limits_{n=1}^\infty n\Gamma(-\frac{1}{2}; 4\pi n^2y)q_{\tau}^{-n^2},
		$$
		where $\tau=x+iy\in \HH$ and $q_\tau:=e^{2\pi i\tau},$ is a weight $3/2$ harmonic Maass form with manageable growth on $\Gamma_0(4).$ 
	\end{theorem}
	
	Furthermore, every weight $k\neq 1$ harmonic weak Maass form $f(\tau)$ has a Fourier expansion of the form
	\begin{equation}\label{HMFFourier}
		f(\tau)=f^{+}(\tau)+\frac{(4\pi y)^{1-k}}{k-1}\overline{c_f^{-}(0)}+f^{-}(\tau),
	\end{equation}
	where 
	\begin{equation}\label{HMSParts}
		f^{+}(\tau)=\sum\limits_{n=m_0}^\infty c_f^{+}(n)q_\tau^n \ \ \
		{\text {\rm and}}\ \ \ 
		f^{-}(\tau)=\sum\limits_{\substack{n=n_0\\ n\neq 0}}^\infty \overline{c_f^{-}(n)}n^{k-1}\Gamma(1-k;4\pi |n| y)q_\tau^{-n}.
	\end{equation}
	
	Note that the function $f^{+}(\tau)$ is called the {\it holomorphic part} of $f$ or {\it mock modular form.} We now recall the Rankin-Cohen bracket operators. 
Let $f$ and $g$ be smooth functions defined on the upper-half complex plane $\HH$, and let $k, l\in\R_{>0}$ and $\nu\in\N_0.$ Then the $\nu$th Rankin-Cohen bracket of $f$ and $g$ is defined by
	\begin{equation}\label{RankinCohenBracket}
		[f,g]_\nu:=\frac{1}{(2\pi i)^\nu}\sum\limits_{r+s=\nu}(-1)^r\binom{k+\nu-1}{s}\binom{l+\nu-1}{r}\frac{d^r}{d\tau^r}f\cdot\frac{d^s}{d\tau^s}g.
	\end{equation}

 	In our purpose, we consider $\nu=1.$ 
	
	\begin{proposition}[Theorem 7.1 of \cite{Cohen}]\label{BracketProposition}
		Let $f$ and $g$ be (not necessarily holomorphic) modular forms of weights $k$ and $l,$ respectively on a congruence subgroup $\Gamma.$ Then the following are true.

		\noindent
		(1) We have that $[f,g]_\nu$ is modular of weight $k+l+2\nu$ on $\Gamma.$ 
		
		\noindent
		(2) If $\gamma\in SL_2(\R),$ then
		under the usual modular slash operator we have
		$$
		[f|_k\gamma,g|_l\gamma]_\nu=([f,g]_\nu)|_{k+l+2\nu}\gamma.
		$$
	\end{proposition}
	
We apply Rankin-Cohen bracket operator on the Zagier's function $\mathcal{H}(\tau)$ and certain univariate theta functions and obtain the weighted class number sums appeared in Theorem \ref{relation-2} as Fourier coefficients of non-holomorphic modular forms.

We now recall holomorphic projection and some of its important facts. Let $f:\HH\rightarrow\C$ be a (not necessarily holomorphic) modular form of weight $k\geq 2$ on a congruence subgroup $\Gamma$ with the Fourier expansion given by
	$$
	f(\tau)=\sum_{n\in\Z}c_f(n,y)q_{\tau}^n,
	$$
	where $\tau=x+iy.$ Let $\{\kappa_1,\ldots, \kappa_M\}$ be the cusps of $\Gamma,$ where  $\kappa_1:=i\infty.$ Moreover, for each $j$ let $\gamma_j\in \SL_2(\Z)$ satisfy
	$\gamma_j\kappa_j=i\infty.$ Moreover, suppose that the following are true:

	\noindent
	(1) There is an $\varepsilon>0$ and a constant $c_0^{(j)}\in\C$ for which
	$$
	f\left(\gamma_j^{-1}w\right)\left(\frac{d\tau}{dw}\right)^{k/2}=c_0^{(j)}+O(\im(w))^{-\varepsilon},
	$$
	for all $j=1,\ldots,M$ and $w=\gamma_j\tau.$

	\noindent
	(2) For all $n>0,$ we have that $c_f(n,y)=O(y^{2-k})$ as $y\rightarrow0.$
	
	Then the {\it holomorphic projection of $f$} is defined as follows:
	\begin{equation}
		(\pi_{\text{hol}}f)(\tau):=c_0+\sum\limits_{n=1}^{\infty}c(n)q_{\tau}^n,
	\end{equation}
	where $c_0=c_0^{(1)}$ and for $n\geq1$
	$$
	c(n)=\frac{(4\pi n)^{k-1}}{(k-2)!}\int_{0}^{\infty}c_f(n,y)e^{-4\pi ny}y^{k-2}dy.
	$$
The next proposition provides the holomorphicity of $\pi_{\text{hol}}(f).$	
\begin{proposition}[Proposition 10.2 of \cite{KO}]\label{HolProjProp}
		If the above hypotheses are true, then for $k>2$ $\pi_{\text{hol}}(f)$ is a weight $k$ holomorphic modular form on $\Gamma.$ 
	\end{proposition}
	
	If $f$ is a harmonic Maass form of weight $k\in \frac{1}{2}\Z$ on $\Gamma_0(N)$ with manageable growth at the cusps and $g$ is a holomorphic modular form of weight $l$ on $\Gamma_0(N),$ then the holomorphic modular form $\pi_{\text{hol}}(f)$ has the following decomposition form 
	\begin{equation}\label{HolomorphicProjectionDecomposition}
		\pi_{\text{hol}}([f,g]_\nu)=[f^{+},g]_\nu+\frac{(4\pi)^{1-k}}{k-1}\overline{c_f^{-}(0)}\pi_{\text{hol}}([y^{1-k},g]_\nu)+\pi_{\text{hol}}([f^{-},g]_\nu).
	\end{equation}

 Thanks to Mertens \cite{mertens} for the following computations that  provide closed formulas of the Fourier expansions of the second and third terms present in the decomposition of $\pi_{\text{hol}}([f,g]_\nu)$ as given in \eqref{HolomorphicProjectionDecomposition}.
	
	\begin{lemma}[Lemma V.1.4 of \cite{mertens}]\label{HolProjExplicit1}
		If the hypotheses above is true and $g(\tau)$ has Fourier expansion
		$g(\tau)=\sum\limits_{n=0}^{\infty}a_g(n)q_{\tau}^n,$
		then we have
		$$
		\frac{(4\pi)^{1-k}}{k-1}\pi_{\text{hol}}([y^{1-k},g]_\nu)=\kappa(k,l,\nu)\cdot \sum\limits_{n=0}^\infty n^{k+\nu-1}a_g(n)q_\tau^n,
		$$
		where
		$$
		\kappa(k,l,\nu):=\frac{1}{(k+l+2\nu-2)!(k-1)}\sum\limits_{\mu=0}^\nu\left(\frac{\Gamma(2-k)\Gamma(l+2\nu-\mu)}{\Gamma(2-k-\mu)}\binom{k+\nu-1}{\nu-\mu}\binom{l+\nu-1}{\mu}\right).
		$$
	\end{lemma}
	
	Let $P_{a,b}(X,Y)$ be the homogeneous polynomial of degree $a-2$ as defined by 
	\begin{align}\label{1}
		P_{a,b}(X,Y):=\sum\limits_{j=0}^{a-2}\binom{j+b-2}{j}X^j(X+Y)^{a-j-2},
	\end{align}
	where $a\geq 2$ is a positive integer and $b$ is any real number. 
	
	\begin{theorem}[Theorem V.1.5 of \cite{mertens}]\label{HolProjExplicit2}
		If  $c_f^{-}(n)$ and $a_g(n)$ are bounded polynomially, then we have
		$\pi_{\text{hol}}([f^{-},g]_\nu=\sum\limits_{r=1}^\infty b(r)q_\tau^{r},$
		where 
		
		\begin{displaymath}
			\begin{split}
				b(r)=-\Gamma(1-k)&\sum\limits_{m-n=r}a_g(m)\overline{c^{-}_f(n)}\sum\limits_{\mu=0}^\nu\binom{k+\nu-1}{\nu-\mu}\binom{l+\nu-1}{\mu}m^{\nu-\mu}\\ 
				&\ \ \ \ \ \ \ \ \ \ \ \ \ \ \ \ \ \ \ \ \times \left(m^{\mu-2\nu-l+1}P_{k+l+2\nu,2-k-\mu}(r,n)-n^{k+\mu-1}\right),
			\end{split}
		\end{displaymath}
		where the sum runs over positive integers $m,n.$
	\end{theorem}

%%%%%%%%%%%%%%%%%%%%%%%%%%%%%%%%%%%%
 
\section{Asymptotic formulas of weighted sums of class numbers}
We are now ready to establish asymptotic formulas for the weighted class number sums. We begin this section by recalling a classical theorem of Eichler.

\begin{theorem}[Eichler \cite{Eichler-1, Eichler-2}]\label{Eichler}
If $n$ is an odd positive integer, then 
$$
\sum\limits_{-\sqrt{n}\leq s\leq\sqrt{n}}H^*(n-s^2)=-\lambda_1(n)+\frac{1}{3}\sigma_1(n),
$$
where $\sigma_1(n):=\sum\limits_{d|n}d$ and $\lambda_1(n):=\frac{1}{2}\sum\limits_{d|n}\min(d,n/d).$
\end{theorem}

We also recall the celebrated theorem of Deligne, which bounds Fourier coefficients of integer weight cusp forms. 
	\begin{theorem}[Remark 9.3.15 of \cite{stromberg}]\label{Deligne}
		Let $f=\sum\limits_{n\geq 1} a(n)q_\tau^n$ be a cusp form of integer weight $k$ on a congruence subgroup. Then for all $\varepsilon>0$ we have
		$a(n)=O_\varepsilon(n^{(k-1)/2+\varepsilon}).$
	\end{theorem}

We are now going to proof Theorems \ref{case-1}--\ref{case-2}.
\begin{proof}[Proof of Theorem \ref{case-1}]
We give the details for the proof when $p\equiv1\pmod{8}.$ For $p\equiv5\pmod{8},$ the proof is completely analogous and we leave it to the reader. Let $p\equiv1\pmod{8}.$ Then we have $s\equiv2\pmod{8}$ and therefore, $s$ is even. Using this we have 
$$
\sum\limits_{s\equiv p+1\pmod{8}}H^*\left(\frac{4p-s^2}{4}\right)s^2=4\sum\limits_{s\equiv 1\pmod{4}}H^*(p-s^2)s^2=2 \sum\limits_{s\in\mathbb{Z}}H^*(p-s^2)s^2.
$$ 

Cohen (see Conjecture I.2.1 of \cite{mertens} and \cite{Cohen}) conjectured that for a positive integer $n,$ the coefficient of $X^{2n}$ in the following expansion
$$
\sum\limits_{\ell \ \text{odd}}\left(\sum\limits_{s\in\Z}\frac{H^*(\ell-s^2)}{1-2sX+\ell X^2}+\sum\limits_{k=0}^{\infty}\lambda_{2k+1}(\ell)X^{2k}\right)q_{\tau}^{\ell}
$$
is a cusp form of weight $2n+2$ on $\Gamma_0(4),$ where $\lambda_{2k+1}(\ell):=\frac{1}{2}\sum\limits_{d|\ell}\min(d,\ell/d)^{2k+1}.$ This conjecture was proved by Mertens \cite{mertens-2}. By Lemma 7.5 of \cite{Cohen} one can write the Fourier coefficients of 
$X^n$ for each $n\geq1.$ For instance, the coefficient of 
$X^{2}$ is given by
$$
\sum\limits_{\ell \ \text{odd}}\left(4\sum\limits_{s\in\Z}H^{*}(\ell-s^2)s^2-\ell\sum\limits_{s\in\Z}H^*(\ell-s^2)+\lambda_3(\ell)\right)q_{\tau}^{\ell}.
$$
Using Deligne's theorem (Theorem \ref{Deligne}) we have 
$$
4\sum\limits_{s\in\Z}H^{*}(p-s^2)s^2-p\sum\limits_{s\in\Z}H^*(p-s^2)=O_{\epsilon}(p^{\frac{3}{2}+\epsilon})
$$
for all $\epsilon>0.$
Finally, using Eichler's theorem (Theorem \ref{Eichler}) we obtain the result.

\end{proof}

We need the following lemma to prove Theorem \ref{case-3}.
\begin{lemma}\label{counting-1}
Let $p>5$ be a prime such that $p\equiv1\pmod{4}.$ Then for all $\epsilon>0$
$$
12\sum\limits_{s\equiv p+1\pmod {16}}H^\ast\left(\frac{4p-s^2}{16}\right)=6\sum\limits_{s\in \Z}H^\ast\left(\frac{4p-s^2}{16}\right)=\frac{p}{2}+O_{\epsilon}(p^{\epsilon}).
$$
\end{lemma}
\begin{proof}
The first equality is trivial. We now prove for all $\epsilon>0$ $$12\sum\limits_{s\equiv p+1\pmod {16}}H^\ast\left(\frac{4p-s^2}{16}\right)=\frac{1}{2}\sum\limits_{\lambda\in\F_p\setminus\{0,\pm1\}}(1+\phi(1-\lambda^2))=\frac{p}{2}+O_{\epsilon}(p^{\epsilon}).$$
By Schoof's theorem (Theorem \ref{Schoof}) it can be shown that $H^*\left(\frac{4p-s^2}{16}\right)$ counts the number of isomorphism classes of elliptic curves $E/\F_p$ with $|E(\F_p)|=p+1-s$ such that $\Z/4\Z\times \Z/4\Z\subseteq E(\F_p).$ This gives $\Z/2\Z\times \Z/4\Z\subseteq E(\F_p).$ Then by Proposition \ref{isomorphism-1} there exist $\lambda\in\F_p\backslash\{0,\pm1\}$ such that $E^{\Leg}_{\lambda^2}$ is isomorphic over $\F_p$ to the curve $E$ and by Proposition \ref{isomorphism-2} we have that $1-\lambda^2=\square.$ 
Thus, using Lemma \ref{final-count-1} we have 
$$
\frac{1}{2}\sum\limits_{\lambda\in\F_p\setminus\{0,\pm1\}}(1+\phi(1-\lambda^2))=12\sum\limits_{s\equiv p+1\pmod{16}}H^*\left(\frac{4p-s^2}{16}\right).
$$
Furthermore, it is routine to verify that for all $\epsilon>0$

\begin{align}
\sum\limits_{\lambda\in\F_p\setminus\{0,\pm1\}}(1+\phi(1-\lambda^2))&=p+O_{\epsilon}(p^{\epsilon}).\notag
\end{align}
Hence, the result follows.
\end{proof}

\begin{proof}[Proof of Theorem \ref{case-3}]
We have 
$$
12\sum\limits_{s\equiv p+1\pmod {16}}H^\ast\left(\frac{4p-s^2}{16}\right)s^2=6\sum\limits_{s\in\Z}H^\ast\left(\frac{4p-s^2}{16}\right)s^2.
$$
Let $f(\tau):={\mathcal{H}}(16\tau) $ and $g(\tau):=\sum\limits_{n\in\Z}\leg{n}{-4}q_{\tau}^{n^2},$ where $\leg{\cdot}{-4}$ is Kronecker symbol. Then, we consider $\pi_{\text{hol}}([f,g]_{1}).$ Now, by using \eqref{HolomorphicProjectionDecomposition} (with $\nu=1$), Lemma  V.2.6 of \cite{mertens}, Lemma \ref{HolProjExplicit1},  Theorem \ref{HolProjExplicit2} and Proposition V.2.7 of \cite{mertens} we obtain the $n$-th Fourier coefficient of $\pi_{\text{hol}}([f,g]_{1})$ and then  applying Deligne's Theorem (Theorem \ref{Deligne}) we have that

\begin{align}\label{Fourier-1}
&\sum\limits_{\mu=0}^{1}(-1)^{1-\mu}{{3/2}\choose{\mu}}{{1/2}\choose{1-\mu}}\sum\limits_{s\in \Z}\leg{s}{-4} s^{2\mu}(n-s^2)^{1-\mu}H^\ast\left(\frac{n-s^2}{16}\right)\notag\\
&+\frac{1}{16}\delta(n)\leg{\sqrt{n}}{-4}n^{3/2}+\frac{1}{8}\sum\limits_{\substack{t^2-9l^2=n \\ t,l\geq 1}}\leg{t}{-4}(t-3l)^{3}=O_{\epsilon}(n^{\frac{3}{2}+\epsilon}),
\end{align}
 for all $\epsilon>0$ and where $\delta(n):=\begin{cases}
1, & \text{if}\ n \ \text{is \ a \  square}\\
0, & \ \text{otherwise.}
\end{cases}$

\noindent Since we have $\sum\limits_{\substack{t^2-9l^2=n \\ t,l\geq 1}}\leg{t}{-3}(t-3l)^{3}=O_{\epsilon}(n^{\frac{3}{2}+\epsilon})$ for all $\epsilon>0.$ Therefore, using this in \eqref{Fourier-1} and considering $n=4p$, where $p\equiv1\pmod{4},$ we have

\begin{align}\label{Fourier-3}
&\sum\limits_{\mu=0}^{1}(-1)^{1-\mu}{{3/2}\choose{\mu}}{{1/2}\choose{1-\mu}}\sum\limits_{s\in \Z} s^{2\mu}(4p-s^2)^{1-\mu}H^\ast\left(\frac{4p-s^2}{16}\right)=O_{\epsilon}(p^{\frac{3}{2}+\epsilon})
\end{align}
for all $\epsilon>0.$
Now, expanding the above sum we have 
\begin{align}
-\frac{1}{2}\sum\limits_{s\in \Z} H^\ast\left(\frac{4p-s^2}{16}\right)(4p-s^2)  +\frac{3}{2}\sum\limits_{s\in \Z}H^\ast\left(\frac{4p-s^2}{16}\right) s^2=O_{\epsilon}(p^{\frac{3}{2}+\epsilon}).
\end{align}
Thus we have 

$$
6\sum\limits_{s\in \Z}H^\ast\left(\frac{4p-s^2}{16}\right) s^2=6p\sum\limits_{s\in \Z}H^\ast\left(\frac{4p-s^2}{16}\right)+O_{\epsilon}(p^{\frac{3}{2}+\epsilon}).
$$
Finally, applying Lemma \ref{counting-1} we obtain that
$$
6\sum\limits_{s\in \Z}H^\ast\left(\frac{4p-s^2}{16}\right) s^2=p\left(\frac{p}{2}+O_{\epsilon}(p^{\epsilon})\right)+O_{\epsilon}(p^{\frac{3}{2}+\epsilon})=\frac{p^2}{2}+O_{\epsilon}(p^{\frac{3}{2}+\epsilon}).
$$

\end{proof}

\begin{proof}[Proof of Theorem \ref{case-2}]
Here we prove the result for $p\equiv3\pmod{8}.$ The proof for the case $p\equiv7\pmod{8}$ follow in a similar manner.
Now, if $p\equiv3\pmod{8},$ then it is trivial to see that 
\begin{align}\label{20may-eqn3}
\sum\limits_{s\equiv p+1\pmod{8}}H^*\left(\frac{4p-s^2}{4}\right)s^2&=\sum\limits_{s\equiv 4\pmod{8}}H^*\left(\frac{4p-s^2}{4}\right)s^2=4\sum\limits_{s\equiv 2\pmod{4}}H^*\left(p-s^2\right)s^2\notag\\
&=4\sum\limits_{s\in\Z}H^*\left(p-s^2\right)s^2-64\sum\limits_{s\in\Z}H^*\left(p-16s^2\right)s^2.
\end{align}
Following the proof of Theorem \ref{case-1}, we have that
\begin{align}\label{20may-eqn-5}
\sum\limits_{s\in\Z}H^*\left(p-s^2\right)s^2=\frac{p^2}{12}+O_{\epsilon}(p^{\frac{3}{2}+\epsilon}).
\end{align}

Let $f(\tau)=\mathcal{H}(\tau)$ and $g(\tau)=\theta(16\tau)=\sum\limits_{n\in\Z}q_{\tau}^{16n^2}.$ Then following similar steps as shown in the proof of Theorem \ref{case-3} we derive the $p$-th Fourier coefficient of $\pi_{\text{hol}}([f,g]_{1})$ and then using Deligne's Theorem (Theorem \ref{Deligne}) we obtain the following asymptotic formula:

\begin{align}\label{20may-eqn-1}
2\times 16\sum\limits_{s\in\Z}H^*\left(p-16s^2\right)s^2-\frac{p}{2}\sum\limits_{s\in\Z}H^*\left(p-16s^2\right)=O_{\epsilon}(p^{\frac{3}{2}+\epsilon})
\end{align}
for all $\epsilon>0.$
By using Schoof's theorem (Theorem \ref{Schoof}) and Lemma \ref{final-count-2}, if we count the number of isomorphism classes of Legendre elliptic curves containing  $\Z/2\Z\times\Z/4\Z,$ then for all $\epsilon>0$
\begin{align}\label{20may-eqn2}
4\sum\limits_{s\equiv p+1\pmod{8}}H^*\left(\frac{4p-s^2}{4}\right)=p+O_{\epsilon}(p^{\epsilon}).
\end{align}
Moreover, for $p\equiv3\pmod{8},$ it is clear that 
$$
\sum\limits_{s\equiv p+1\pmod{8}}H^*\left(\frac{4p-s^2}{4}\right)=\sum\limits_{s\equiv 2\pmod{4}}H^*(p-s^2)=\sum\limits_{s\in \Z}H^*(p-s^2)-\sum\limits_{4|s}H^*(p-s^2).
$$
Now, using \eqref{20may-eqn2} in the above we have for all $\epsilon>0$

$$
\sum\limits_{4|s}H^*(p-s^2)=\sum\limits_{s\in \Z}H^*(p-s^2)-\frac{p}{4}+O_{\epsilon}(p^{\epsilon}).
$$
Hence, using  Eichler's theorem (Theorem \ref{Eichler}) we have for all $\epsilon>0$
$$
\sum\limits_{4|s}H^*(p-s^2)=\sum\limits_{s\in\Z}H^*(p-16s^2)=\frac{p}{12}+O_{\epsilon}(p^{\epsilon}).
$$
Substituting the above identity into \eqref{20may-eqn-1} we obtain for all $\epsilon>0$

 \begin{align}\label{20may-eqn-6}
 64\sum\limits_{s\in\Z}H^*\left(p-16s^2\right)s^2=\frac{p^2}{12}+O_{\epsilon}(p^{\frac{3}{2}+\epsilon}).
 \end{align}

\noindent Finally, substituting \eqref{20may-eqn-5} and \eqref{20may-eqn-6} into \eqref{20may-eqn3} we obtain the required result.
\end{proof}

\subsection{Proof of Corollary \ref{twisted-moment-1} and Corollary \ref{moment-1}}

\begin{proof}[Proof of Corollary \ref{twisted-moment-1}]
First we prove the claim for $p\equiv1\pmod{4}.$ Using first part of Theorem \ref{relation-2}, Theorem \ref{case-1} and Theorem \ref{case-3} we obtain the result. Similarly, we prove the identity
for $p\equiv3\pmod{4}$ by using second part of Theorem \ref{relation-2} and Theorem \ref{case-2}. This completes the proof.
\end{proof}

\begin{proof}[Proof of Corollary \ref{moment-1}]
Using the relations $\pi(a)+\overline{\pi}(a)=-K(a,p)$ and $\pi(a)\overline{\pi}(a)=p$ we may write

%$$\pi(a)^4+\pi(a)^3\overline{\pi}(a)+\pi(a)^2\overline{\pi}(a)^2+\pi(a)\overline{\pi}(a)^3+\overline{\pi}(a)^4=K(a)^4-3pK(a)^2+p^2.$$ 

%This gives

\begin{align}\label{eqn-3}
M(4,\phi)_p&=\sum_{a\in\mathbb{F}_p^{\times}}\phi(a)(K(a,p)^4-3pK(a,p)^2+p^2)\notag\\
&=S(4,\phi)_p-3pS(2,\phi)_p.
\end{align}
Now, applying Proposition \ref{proposition-1} with $m=1,$ we have

\begin{align}\label{eqn-5}
S(2,\phi)_p&=p\phi(-1)\sum_{x\in\mathbb{F}_p^{\times}}\phi(x+1)\phi(\overline{x}+1)
=p\phi(-1)\sum_{x\in\mathbb{F}_p^{\times}, x\neq-1}\phi(x)=-p.
\end{align}
Now, substituting \eqref{eqn-5} into \eqref{eqn-3}, we have 

\begin{align}\label{pf-eq-1}
 M(4,\phi)_p=S(4,\phi)_p+3p^2.   
\end{align}
Finally, applying Corollary \ref{twisted-moment-1} we deduce the result. This completes the proof.
\end{proof}

\section{Hypergeometric functions}\label{hypergeometric}
In this section, we discuss certain relations of $S(4,\phi)_p$ involving $p$-adic hypergeometric functions over finite fields and present two results establishing asymptotic weighted averages of two families of $p$-adic hypergeometric functions. To state our results precisely, we now fix some notation and recall some basic definitions.
Let $\mathbb{Z}_p$ denote the ring of $p$-adic integers and $\mathbb{Q}_p$ denote the field of $p$-adic numbers. 
For a positive integer $n,$
the $p$-adic gamma function $\Gamma_p(n)$ is defined as
\begin{align}
\Gamma_p(n):=(-1)^n\prod\limits_{0<j<n,p\nmid j}j.\notag
\end{align}
 It can be extended to all $x\in\mathbb{Z}_p$ by setting $\Gamma_p(0):=1$ and for $x\neq0$
\begin{align}
\Gamma_p(x):=\lim_{x_n\rightarrow x}\Gamma_p(x_n),\notag
\end{align}
where $(x_n)$ is a sequence of positive integers $p$-adically approaching to $x.$
For further details, see \cite{kob}. 
Let $\omega$ be the Teichm\"{u}ller character of $\mathbb{F}_p$ with $\omega(a)\equiv a\pmod{p}.$ For $x\in\mathbb{Q},$ $\lfloor x\rfloor$ denotes the greatest integer less than or equal to $x$ and $\langle x\rangle$ denotes the fractional part of $x$, satisfying $0\leq\langle x\rangle<1$. Using these notation $p$-adic hypergeometric function over finite field is given as follows:
\begin{definition}\cite[Definition 5.1]{mccarthy2} \label{defin1}
Let $p$ be an odd prime and $t \in \mathbb{F}_p$.
For a positive integer $n$ and $1\leq k\leq n$, let $a_k$, $b_k$ $\in \mathbb{Q}\cap \mathbb{Z}_p$.
Then 
\begin{align}
&{_n\mathbb{G}_n}\left(\begin{array}{cccc}
             a_1, & a_2, & \ldots, & a_n \\
             b_1, & b_2, & \ldots, & b_n
           \end{array}\mid t
 \right)_p:=\frac{-1}{p-1}\sum_{a=0}^{p-2}(-1)^{an}~~\overline{\omega}^a(t)\notag\\
&\times \prod\limits_{k=1}^n(-p)^{-\lfloor \langle a_k \rangle-\frac{a}{p-1} \rfloor -\lfloor\langle -b_k \rangle +\frac{a}{p-1}\rfloor}
 \frac{\Gamma_p(\langle a_k-\frac{a}{p-1}\rangle)}{\Gamma_p(\langle a_k \rangle)}
 \frac{\Gamma_p(\langle -b_k+\frac{a}{p-1} \rangle)}{\Gamma_p(\langle -b_k \rangle)}.\notag
\end{align}
\end{definition}
\noindent In our purpose, for $\lambda\in\F_p,$ we study two families of these functions, namely, 
$${_3G_3}(\lambda)_p:={_3\mathbb{G}_3}\left(\begin{matrix}
\frac{5}{6} &\frac{1}{12} , & \frac{7}{12} \vspace{1mm}\\
\frac{1}{3}  & \frac{1}{3} & \frac{1}{3}
\end{matrix}\mid\frac{\lambda-1}{\lambda}\right)_p$$ and  $${_9G_9}(\lambda)_p:={_9\mathbb{G}_9}\left(\begin{matrix}
\frac{1}{12} & \frac{1}{6} & \frac{1}{4} & \frac{5}{12} & \frac{1}{2} & \frac{7}{12} & \frac{3}{4} & \frac{5}{6} &  \frac{11}{12} \vspace{1mm}\\
\frac{1}{3}  & \frac{1}{3} & \frac{1}{3} & \frac{2}{3} & \frac{2}{3}  & \frac{2}{3}  & 0 & 0 & 0 
\end{matrix}\mid \lambda \right)_p.$$
For brevity, we use these two notation in the rest of the paper. In the next two theorems, we compute the asymptotic averages of these two families of hypergeometric functions twisted by a character.

\begin{theorem}\label{3G3-moment}
Let $p\equiv1\pmod{3}$ be a prime and $\psi_6=\omega^{\frac{(p-1)}{6}}$ be a character of order 6. Then as $p\rightarrow\infty$ the following holds:

\begin{equation}
\sum\limits_{\lambda\in\F_p}\psi_6\left(\lambda(1-\lambda)^2\right)\cdot {_3G_3}(\lambda)_p=o(1).\nonumber
\end{equation}

\end{theorem}

\begin{theorem}\label{9G9-moment}
 Let $p\equiv2\pmod{3}$ be an odd prime. Then as $p\rightarrow\infty$ the following holds: \\

$$
\sum\limits_{\lambda\in\F_p\setminus\{0\}}\phi(\lambda^{1/3}-1)\cdot {_9G_9}(\lambda)_p=o(p^2),
$$

where $\lambda^{1/3}$ is the unique solution of $x^3\equiv\lambda\pmod{p}.$

\end{theorem}

%%%%%%%%%%%%%%%%%%%%%
To proceed further, we first recall some important theorems, lemmas and definitions that are very useful to prove the results. For multiplicative characters $\chi$ and $\psi$ of 
$\mathbb{F}_p$ (with the convention that $\chi(0)=0$ for multiplicative characters), Jacobi sum $J(\chi,\psi)$ is defined by
\begin{align}
J(\chi,\psi):=\sum_{y\in\mathbb{F}_p}\chi(y)\psi(1-y).\notag
\end{align}
Moreover, the normalized Jacobi sum also known as a binomial coefficient ${\chi\choose \psi}$ is defined by
$$
{\chi\choose \psi}:=\frac{\psi(-1)}{p}J(\chi,\overline{\psi}).
$$
Using these binomial coefficients, Greene \cite{greene} introduced a class of hypergeometric function over finite fields which are known as Gaussian hypergeometric functions. For example, the following function is one of the simplest and most studied hypergeometric functions over finite fields:

$$
{_2F_1}(\lambda)_p:={_2F_1}\left(\begin{matrix}
\phi & \phi\\
~ & \varepsilon
\end{matrix}\mid\lambda\right)_p:=\frac{p}{p-1}\sum\limits_{\chi\in\widehat{\F_p^{\times}}}{\phi\chi\choose \chi}{\phi\chi\choose \chi}\chi(\lambda).
$$

In addition to their number-theoretic significance, Greene’s hypergeometric functions can be naturally viewed as finite field analogues of classical hypergeometric functions (for example, see \cite{evans2, EG, greene} etc). Among various results in this area, Ono \cite{ono} demonstrated that the Frobenius traces of Legendre elliptic curves can be expressed as special values of $_2F_1$-hypergeometric functions. More precisely, for $\lambda\neq0,1,$ he showed that:
\begin{align}\label{Trace-relation}
{_2F_1}(\lambda)_p=\frac{-\phi(-1)}{p} \cdot a_p(\lambda).
\end{align}

In \cite{KHN}, the authors initiated a study of limiting distribution of this family of hypergeometric functions over random large finite fields confirming the semicircular or Sato-Tate distribution by developing asymptotic moments of ${_2F_1}(\lambda)_p.$
Using the relation \eqref{Trace-relation}, \cite[Proposition 3]{ono} and the moment formula, namely Theorem 1.1 of \cite{KHN} we may express

\begin{align}
\sum\limits_{\lambda\neq0,1}a_p(\lambda)^2=p^2+o(p^2).\notag
\end{align}

Moreover, if we use \cite[Theorem 2]{ono}, then the above asymptotic identity is equivalent to 

\begin{align}\label{22Sep-eqn-2}
\sum\limits_{\lambda\neq0,\pm1}a_p(\lambda)^2=p^2+o(p^2).
\end{align}

Let $\zeta_p$ denote a fixed primitive $p$-th root of unity. For a multiplicative character
 $\chi$ of $\mathbb{F}_p,$ Gauss sum is defined by
\begin{align}
g(\chi):=\sum\limits_{x\in \mathbb{F}_p}\chi(x)~\zeta_p^x.\notag
\end{align}

Let $\pi \in \mathbb{C}_p$ be the fixed root of the polynomial $x^{p-1} + p$, which satisfies the congruence condition
$\pi \equiv \zeta_p-1 \pmod{(\zeta_p-1)^2}$. We now provide an important formula due to Gross and Koblitz that relates Gauss sums with $p$-adic gamma functions.
\begin{theorem}\cite[Gross-Koblitz formula]{gross}\label{gross-koblitz} For $j\in \mathbb{Z}$,
\begin{align}
g(\overline{\omega}^j)=-\pi^{(p-1)\langle\frac{j}{p-1} \rangle}\Gamma_p\left(\left\langle \frac{j}{p-1} \right\rangle\right).\notag
\end{align}
\end{theorem}

To this end, we discuss certain product formulas of $p$-adic gamma functions. For details, see \cite{mccarthy2}. If $m\in\mathbb{Z}^+,$ 
$p\nmid m$  and $x=\frac{r}{p-1}$ with $0\leq r\leq p-1,$ then
\begin{align}\label{prod-1}
\prod_{h=0}^{m-1}\Gamma_p\left(\frac{x+h}{m}\right)=\omega(m^{(1-x)(1-p)})~\Gamma_p(x)\prod_{h=1}^{m-1}\Gamma_p\left(\frac{h}{m}\right).
\end{align}
If $t\in\mathbb{Z}^{+}$ and $p\nmid t,$ then for $0\leq j\leq p-2$ we have

\begin{align}\label{new-prod-1}
\omega(t^{tj})\Gamma_p\left(\left\langle\frac{tj}{p-1}\right\rangle\right)\prod_{h=1}^{t-1}\Gamma_p\left(\frac{h}{t}\right)
=\prod_{h=0}^{t-1}\Gamma_p\left(\left\langle\frac{h}{t}+\frac{j}{p-1}\right\rangle\right)
\end{align}
and
\begin{align}\label{prod-2}
\omega(t^{-tj})\Gamma_p\left(\left\langle\frac{-tj}{p-1}\right\rangle\right)\prod_{h=1}^{t-1}\Gamma_p\left(\frac{h}{t}\right)
=\prod_{h=1}^{t}\Gamma_p\left(\left\langle\frac{h}{t}-\frac{j}{p-1}\right\rangle\right).
\end{align}

Furthermore, we recall the Hasse-Davenport formula for Gauss sums.
\begin{theorem}\cite[Hasse-Davenport relation, Theorem 11.3.5]{berndt}
Let $\chi$ be a multiplicative character of order $m$ of $\mathbb{F}_p$ for some positive integer $m.$ For a multiplicative character $\psi$ of $\mathbb{F}_p,$ we  have
\begin{align}\label{dh}
\prod_{i=0}^{m-1}g(\psi\chi^i)=g(\psi^m)\psi^{-m}(m)\prod_{i=1}^{m-1}g(\chi^i).
\end{align}
\end{theorem}

We now discuss three propositions that are going to use in the proofs of Theorem \ref{3G3-moment} and Theorem \ref{9G9-moment}.

\begin{proposition}\label{hypergeometric-1}

Let $p$ be an odd prime such that $p\equiv1\pmod{3}$. Then we have
\begin{align*}
&\sum\limits_{\lambda\in\F_p}\phi(\lambda)\sum_{a=0}^{p-2}g(\varphi\omega^{a})g(\overline{\omega}^a)^3g(\varphi\omega^{2a})~\overline{\omega}^a\left(\frac{4(1-\lambda)}{\lambda}\right)\\&=p^3(p-1)\psi_6(-2)\phi(2)~\sum\limits_{\lambda\in\F_p}\psi_6(\lambda(1-\lambda^2))\cdot {_3G_3}(\lambda)_p.
\end{align*}

\end{proposition}

\begin{proof}

Let $I=\sum\limits_{\lambda\in\F_p}\phi(\lambda)\sum\limits_{a=0}^{p-2}g(\varphi\omega^{a})g(\overline{\omega}^a)^3g(\varphi\omega^{2a})~\overline{\omega}^a\left(\frac{4(1-\lambda)}{\lambda}\right)$.
If we replace $a$ by $a-\frac{p-1}{3},$ and then applying Gross-Koblitz formula (Theorem~\ref{gross-koblitz}) on each of the Gauss sums we may write

\begin{align}\label{eqn-prop-1}
I&=-\sum\limits_{\lambda\in\F_p}\phi(\lambda)\psi_3\left(\frac{4(1-\lambda)}{\lambda}\right)
\sum\limits_{a=0}^{p-2}\pi^{(p-1)e_a}~\overline{\omega}^{a}\left(\frac{4(1-\lambda)}{\lambda}\right)
~\Gamma_p\left(\left\langle\frac{5}{6}-\frac{a}{p-1}\right\rangle\right)\notag\\
&\times \Gamma_p\left(\left\langle\frac{2}{3}+\frac{a}{p-1}\right\rangle\right)^3\Gamma_p\left(\left\langle\frac{1}{6}-\frac{2a}{p-1}\right\rangle\right),
\end{align}

where $e_a=\left\langle\frac{5}{6}-\frac{a}{p-1}\right\rangle+3\left\langle\frac{2}{3}+\frac{a}{p-1}\right\rangle
+\left\langle\frac{1}{6}-\frac{2a}{p-1}\right\rangle$ and $\psi_3=\omega^{\frac{p-1}{3}}$ is a cubic character.

 If we apply \eqref{prod-1} for $m=2$ and 
$x=\left\langle\frac{1}{6}-\frac{2a}{p-1}\right\rangle$, then we obtain

\begin{align}\label{22Sep-eq-1}
\Gamma_p\left(\left\langle\frac{1}{6}-\frac{2a}{p-1}\right\rangle\right)
=\frac{\omega^a(4)\overline{\psi}_6(2)}{\Gamma_p(\frac{1}{2})}
\Gamma_p\left(\left\langle\frac{1}{12}-\frac{a}{p-1}\right\rangle\right)
\Gamma_p\left(\left\langle\frac{7}{12}-\frac{a}{p-1}\right\rangle\right).
\end{align}

It is easy to see that $\left\lfloor\frac{1}{6}-\frac{2a}{p-1}\right\rfloor=\left\lfloor\frac{1}{12}-\frac{a}{p-1}\right\rfloor
+\left\lfloor\frac{7}{12}-\frac{a}{p-1}\right\rfloor.$ 
Further calculating the exponent $e_a$ by using the above identity we have 
$$e_a=3-\left\lfloor\frac{5}{6}-\frac{a}{p-1}\right\rfloor-\left\lfloor\frac{1}{12}-\frac{a}{p-1}\right\rfloor
-\left\lfloor\frac{7}{12}-\frac{a}{p-1}\right\rfloor
-3\left\lfloor\left\langle-\frac{1}{3}\right\rangle+\frac{a}{p-1}\right\rfloor.$$

Now, Substituting the above expression into \eqref{eqn-prop-1} and using \eqref{22Sep-eq-1}  we derive that

$$I=-C\cdot p^3(p-1)\phi(2)~\sum\limits_{\lambda\in\F_p}\psi_6(\lambda(1-\lambda^2))\cdot {_3G_3}(\lambda)_p,$$

where $C=\frac{\Gamma_p(\frac{2}{3})^3\Gamma_p(\frac{5}{6})\Gamma_p(\frac{1}{12})
\Gamma_p(\frac{7}{12})}{\Gamma_p(\frac{1}{2})}.$ It is easy to see that $C=-\psi_6(-2).$ This completes the proof.
\end{proof}

\begin{proposition}\label{hypergeometric-2}

Let $p$ be an odd prime such that $p\equiv2\pmod{3}$. Then we have
\begin{align*}
&\sum\limits_{\lambda\in\F_p}\phi(\lambda)\sum_{a=0}^{p-2}g(\phi\omega^{a})g(\overline{\omega}^a)^3g(\phi\omega^{2a})~\overline{\omega}^a\left(\frac{4(1-\lambda)}{\lambda}\right)\\&=
p(p-1)\sum\limits_{\lambda\in\F_p}\phi(\lambda^{1/3}-1)\cdot {_9G_9}(\lambda)_p
\end{align*}

where $\lambda^{1/3}$ represents the unique solution to $x^3\equiv\lambda\pmod{p}.$
\end{proposition}

\begin{proof}
Let $I=\sum\limits_{\lambda\in\F_p}\phi(\lambda)\sum\limits_{a=0}^{p-2}g(\phi\omega^{a})g(\overline{\omega}^a)^3g(\phi\omega^{2a})~\overline{\omega}^a\left(\frac{4(1-\lambda)}{\lambda}\right).$ Now, taking the transformation $a\rightarrow3a$ we have

$$
I=\sum\limits_{\lambda\in\F_p}\phi(\lambda)\sum\limits_{a=0}^{p-2}g(\phi\omega^{3a})g(\overline{\omega}^{3a})^3g(\phi\omega^{6a})~\overline{\omega}^{3a}\left(\frac{4(1-\lambda)}{\lambda}\right).
$$

Applying Davenport Hasse relation for $m=2$ and $\psi=\omega^{3a}, \omega^{6a}$ successively, we have

\begin{align}
g(\phi\omega^{3a})&=\frac{g(\omega^{6a})\omega^{3a}(2^{-2})g(\phi)}{g(\omega^{3a})}\ \text{and}\ g(\phi\omega^{6a})=\frac{g(\omega^{12a})\omega^{6a}(2^{-2})g(\phi)}{g(\omega^{6a})}.\notag
\end{align}

\noindent Using these two expressions and applying Gross Koblitz formula (Theorem \ref{gross-koblitz}) we deduce that

\begin{align}
I&=-p\phi(-1)\sum\limits_{\lambda\in\F_p\setminus\{0\}}\phi(\lambda)\sum_{a=0}^{p-2}\overline{\omega}^{3a}\left(\frac{4^4(1-\lambda)}{\lambda}\right)\pi^{(p-1)e_a}~
\frac{\Gamma_p(\langle\frac{3a}{p-1}\rangle)^3
\Gamma_p(\langle-\frac{12a}{p-1}\rangle)}{\Gamma_p(\langle-\frac{3a}{p-1}\rangle)},
\end{align}

where $e_a=3\langle\frac{3a}{p-1}\rangle+\langle\frac{-12a}{p-1}\rangle
-\langle\frac{-3a}{p-1}\rangle.$
Now by applying \eqref{new-prod-1}, \eqref{prod-2} and simplifying further we may write

\begin{align}
I&=p(p-1)\phi(-1)\sum\limits_{\lambda\in\F_p\setminus\{0\}}\phi(\lambda)\cdot {_9G_9}\left(\frac{(\lambda-1)^3}{\lambda^3}\right)_p.\notag
\end{align}

Now, replacing $\lambda$ by $-\lambda$ and then $\lambda$ by $1/\lambda$  we have 

 $$
I=p(p-1)\sum\limits_{\lambda\in\F_p\setminus\{0\}}\phi(\lambda)\cdot {_9G_9}((\lambda+1)^3)_p.
$$
Finally, replacing $\lambda$ by $\lambda-1$ and using the fact that $\lambda\rightarrow \lambda^3$ is a bijection we conclude the result.

\end{proof}

\begin{proposition}\label{hypergeometric-3}
Let $p>3$ be a prime. Then we have 

\begin{align}
\frac{S(4,\phi)_p}{p}&=p^2\phi(2)\cdot {_2F_1}(1/2)_p^2-p^2\phi(-1)\cdot{_2F_1}(-1)_p^2-p^2\phi(-2)\cdot{_3F_2}(1)_p\notag\\
&+\frac{1}{p(p-1)}\sum\limits_{\lambda\in\F_p}\phi(\lambda)\sum_{a=0}^{p-2}g(\varphi\omega^{a})g(\overline{\omega}^a)^3g(\varphi\omega^{2a})~\overline{\omega}^a\left(\frac{4(1-\lambda)}{\lambda}\right)+o(p^2).\notag
\end{align}
\end{proposition}

\begin{proof}
By Proposition \ref{proposition-hypergeometric} and \eqref{22Sep-eqn-2} we may write

\begin{align}\label{final-eq-6}
\frac{S(4,\phi)_p}{p}=\sum\limits_{\lambda\neq0,\pm1}\phi(\lambda)a_p(\lambda)^2+o(p^2).    
\end{align}

Now, using \eqref{Trace-relation} and replacing $\lambda$ by $\frac{1-t}{2}$ we express

%%%%%%%%%%%%%%%%%%%%%%%%%%%%%%%%%

\begin{align}\label{eqn-6}
\sum\limits_{\lambda\neq0,\pm1}\phi(\lambda)a_p(\lambda)^2
&=p^2\varphi(2)~{_2F_1}\left(\frac{1}{2}\right)_p^2-p^2\phi(-1){_2F_1}(-1)_p^2+
p^2\cdot A,
\end{align}
where $A=\varphi(2)\sum\limits_{t\in\mathbb{F}_p^{\times}, t\neq\pm1}\varphi(1-t)~{_2F_1}\left(\frac{1-t}{2}\right)_p^2.$ Then using a special case of Theorem 1.7 of \cite{EG} and elementary calculation we obtain

\begin{align}
A&=-\frac{1}{p}-\frac{\varphi(2)}{p}+\frac{\varphi(-2)p}{p-1}\sum_{t\in\mathbb{F}_p^{\times}, t\neq\pm1}\varphi(1+t)
\sum_{\chi\in\widehat{\mathbb{F}_p^{\times}}}{\varphi\chi\choose\chi}^3\overline{\chi}(1-t^2)\notag.
\end{align}
Then adjusting the term under summation for $t=0,\pm1$ we can write
\begin{align}\label{eqn-7}
A&=-\frac{1}{p}-\frac{\varphi(2)}{p}-\varphi(-2)\cdot{_3F_2}(1)_p+\frac{\varphi(-2)p}{p-1}\sum_{t\in\mathbb{F}_p}\varphi(1+t)
\sum_{\chi\in\widehat{\mathbb{F}_p^{\times}}}{\varphi\chi\choose\chi}^3\overline{\chi}(1-t^2).
\end{align}
Now, consider the sum

\begin{align}
B&=\sum_{t\in\mathbb{F}_p}\varphi(1+t)
\sum_{\chi\in\widehat{\mathbb{F}_p^{\times}}}{\varphi\chi\choose\chi}^3\overline{\chi}(1-t^2)=\sum_{\chi\in\widehat{\mathbb{F}_p^{\times}}}{\varphi\chi\choose\chi}^3\sum_{t\in\mathbb{F}_p}
\varphi \overline{\chi}(1+t)\overline{\chi}(1-t)\notag\\
&=\varphi(2)\sum_{\chi\in\widehat{\mathbb{F}_p^{\times}}}{\varphi\chi\choose\chi}^3\overline{\chi}^2(2)J(\varphi\overline{\chi},\overline{\chi})=p\varphi(2)\sum_{\chi\in\widehat{\mathbb{F}_p^{\times}}}{\varphi\chi\choose\chi}^3
{\varphi\overline{\chi}\choose\chi}\chi\left(-\frac{1}{4}\right).\notag
\end{align}

It is easy to see that ${\varphi\overline{\chi}\choose\chi}=\chi(-1){\varphi\chi^2\choose\chi}$. Using this relation in the above sum we have

\begin{align}
B&=p\varphi(2)\sum_{\chi\in\widehat{\mathbb{F}_p^{\times}}}{\varphi\chi\choose\chi}^3{\varphi\chi^2\choose\chi}
\chi\left(\frac{1}{4}\right).\notag
\end{align}

Using \cite[Eq. (2.9)]{greene} and taking $\chi=\omega^a,$ it is straightforward to deduce that

\begin{align}\label{eqn-9}
B&=\frac{\varphi(-2)}{p^4}\sum_{a=0}^{p-2}
\frac{g(\varphi\omega^a)^2g(\overline{\omega}^a)^4g(\varphi\omega^{2a})}{g(\varphi)}~\overline{\omega}^a(4)-\frac{p-1}{p^3}\varphi(-2)\notag\\
&=\frac{\varphi(-2)}{p^4}\sum\limits_{a=0}^{p-2}g(\varphi\omega^a)g(\overline{\omega}^a)^3g(\varphi\omega^{2a})J(\phi\omega^a,\overline{\omega}^a) ~\overline{\omega}^a(4)-\frac{p-1}{p^3}\varphi(-2)\notag\\
&= \frac{\varphi(-2)}{p^4}\sum\limits_{\lambda\in\F_p}\phi(\lambda)\sum\limits_{a=0}^{p-2}g(\varphi\omega^a)g(\overline{\omega}^a)^3g(\varphi\omega^{2a}) ~\overline{\omega}^a\left(\frac{4(1-\lambda)}{\lambda}\right)
-\frac{p-1}{p^3}\varphi(-2).
\end{align}

Now, substituting \eqref{eqn-9} into \eqref{eqn-7} and then again substituting the resultant identity of $A$ into 
\eqref{eqn-6} we deduce the result.

\end{proof}

\begin{proof}[Proof of Theorem~\ref{3G3-moment}]
By Theorem 2 and Theorem 3 of \cite{ono}  we may write 
$$p^2\cdot{_2F_1}(1/2)_p^2=p^2\cdot{_2F_1}(-1)_p^2=
p^2\cdot{_3F_2}(1)_p=o(p^2).$$

\noindent Combining Proposition \ref{hypergeometric-3} and Proposition \ref{hypergeometric-1} we relate $S(4,\phi)_p$ and the weighted average of ${_3G_3(\lambda)_p}.$ Finally, we conclude the result by using the above asymptotic identities and
Corollary \ref{twisted-moment-1}.

\end{proof}

\begin{proof}[Proof of Theorem~\ref{9G9-moment}]
The proof is analogous to the proof of Theorem~\ref{3G3-moment}. The only change is that we use Proposition~\ref{hypergeometric-2} to establish connection between ${_9G_9(\lambda)_p}$ and the character sum that is associated to $S(4,\phi)_p$ as given in Proposition \ref{hypergeometric-3}.
\end{proof}

\section{Acknowledgements}
The author is supported by Science and Engineering Research Board [CRG/2023/003037].


\begin{thebibliography}{99}

\bibitem{Adolphson}
A. Adolphson, {\it On the distribution of angles of Kloosterman sums,} J. Reine Angew. Math. 395 (1989), 214--220.

\bibitem{ahl-ono-1}
S. Ahlgren and K. Ono, {\it Modularity of a certain Calabi-Yau threefold}, Montash. Math. 129 (2000), 177--190.

%\bibitem{ahl-ono-2}
%S. Ahlgren and K. Ono, {\it A Gaussian hypergeometric series evaluation and Ap\'{e}ry number congruences}, J. Reine Angew. Math 518 (2000), 187--212.



\bibitem{berndt}
B. Berndt, R. Evans, and K. Williams, {\it Gauss and Jacobi Sums}, Canadian Mathematical Society Series of Monographs and Advanced Texts,
A Wiley-Interscience Publication, John Wiley \& Sons, Inc., New York, 1998.

\bibitem{KO}
K. Bringmann, A. Folsom, K. Ono, and L. Rolen, {\it Harmonic Maass forms and mock modular forms: Theory
and applications}, Amer. Math. Soc. Colloq. 64, Amer. Math. Soc., Providence, 2017.


\bibitem{ce}
H. Choi and R. Evans, {\it Congruences for sums of powers of Kloosterman sums}, Int. J. Number Theory 3 (2007), no. 1, 105--117.

\bibitem{Cohen}
H. Cohen, {\it Sums involving the values at negative integers of $L$-functions of quadratic characters}, Math. Ann., 217 (1975), 271--285.

\bibitem{stromberg}
H. Cohen and F. Str\"{o}mberg, {\it Modular forms: A classical approach}, Graduate Studies in Mathematics, Vol 179, Amer. Math. Soc., Providence, 2017.

\bibitem{Conrey}
J. Conrey and H. Iwaniec, {\it The cubic moment of central values of automorphic $L$-functions}, Ann. Math. 151 (2000), 1175--1216.

\bibitem{Da}
H. Davenport, {\it On certain exponential sums}, J. Reine Angew. Math. 169 (1933), 158--176.


\bibitem{di}
J. Deshouillers and H. Iwaniec, {\it Kloosterman sums and Fourier coefficients of cusp forms}, Invent. Math. 70 (1982), no. 2, 219--288.

\bibitem{DGP}
E. P. Dummit, A. W. Goldberg, and A. R. Perry, {\it A conjecture of Evans on sums of Kloosterman sums}, Proc. Amer. Math. Soc. 138 (2010), no. 9, 3047--3056.

\bibitem{Eichler-1}
M. Eichler, {\it On the class number of imaginary quadratic fields and the sums of divisors of natural numbers}, J. Indian Math. Soc. 15 (1955), 153--180.

\bibitem{Eichler-2}
M. Eichler, {\it U\"{b}er die Darstellbarke\'{i}t von Modulformen durch Thetareihen}, J. Reine Angew. Math. 195 (1955), 156--171.

\bibitem{evans3}
R. Evans, {\it Hypergeometric ${_3F_2}(1/4)$ evaluations over finite fields and Hecke eigenforms}, Proc. Amer. Math. Soc. 138 (2010), no. 2, 517--531. 

\bibitem{ronald}
R. Evans, {\it Seventh power moments of Kloosterman sums}, Israel J. Math. 175 (2010), 349--362.

\bibitem{evans2} R. Evans and J. Greene, {\it Evaluation of Hypergeometric Functions over Finite Fields},
Hiroshima Math. J. 39 (2009), no. 2, 217--235.

\bibitem{EG} R. Evans and J. Greene, {\it Clausen's theorem and hypergeometric functions over finite fields}, Finite Fields Appl. 15 (2009), 517--531.

%\bibitem{fop}
%S. Frechette, K. Ono, and M. Papanikolas, {\it Gaussian hypergeometric functions and traces of Hecke operators}, Int. Math. Res. Not. (2004), no. 60.

%\bibitem {Fuselier} 
%J. Fuselier, {\it Traces of Hecke operators in level 1 and Gaussian hypergeometric functions},
%Proc. Amer. Math. Soc. 141 (6), (2013), 1871--1881.

%\bibitem{fuselier-thesis}
%J. Fuselier, {\it Hypergeometric functions over finite fields and relations to modular forms and elliptic curves}, ProQuest LLC, Ann Arbor, MI, 2007. Thesis (Ph.D.)--Texas A\&M University. 

%\bibitem {fuselier2} 
%J. Fuselier, \textit{Hypergeometric functions over $\mathbb{F}_p$ and relations to elliptic curve and modular forms}, Proc. Amer. Math. Soc. 138 (2010), 109--123.

%\bibitem{fm} J. Fuselier and D. McCarthy, {\it Hypergeometric type identities in the $p$-adic setting and modular forms},
%Proc. Amer. Math. Soc. 144 (2016), 1493--1508

\bibitem{Geemen}
van Geemen B, Nygaard N, {\it On the geometry and arithmetic of some Siegel modular
threefolds}. J Number Theory 53 (1995), 45--87.

\bibitem{greene}
J. Greene, {\it Hypergeometric functions over finite fields}, Trans. Amer. Math. Soc. 301 (1987), no. 1, 77--101.

\bibitem{greene2}
J. Greene, {\it Character Sum Analogues for Hypergeometric and Generalized Hypergeometric Functions over Finite Fields},
Ph.D. thesis, Univ. of Minnesota, Minneapolis, 1984.


\bibitem{gross}
B. H. Gross and N. Koblitz, {\it Gauss sum and the $p$-adic $\Gamma$-function}, Annals of Mathematics 109 (1979), 569--581.

\bibitem{Hu}
K. Hulek, J. Spandaw, B. van Geemen and D. van Straten, {\it The modularity of the Barth-Nieto quintic and its relatives}, Adv. Geom. 1 (2001), 263--289.

\bibitem{ireland}
K. Ireland and M. Rosen, {\it A Classical Inroduction to Modern Number Theory}, Springer International Edition, Springer, 2005.

\bibitem{Iwanic}
H. Iwaniec, {\it Topics in classical automorphic forms}, AMS Graduate Studies in Math., vol. 17, Amer. Math. Soc., 1997.


\bibitem{Katz-1}
N. Katz, {\it Gauss sums, Kloosterman sums, and monodromy groups}, Annals of Mathematics Studies, vol. 116, Princeton University Press, Princeton, NJ, 1988.

\bibitem{katz}
N. Katz, {\it Exponential Sums and Differential Equations}, Ann. of Math. Stud., vol. 124, Princeton University Press, Princeton, NJ, 1990.

\bibitem{Kloosterman}
H. D. Kloosterman, {\it On the representation of numbers in the form $ax^2 + by^2 + cz^2 + dt^2$}, Acta Math. 49 (1926), 407--464.

\bibitem{kob} N. Koblitz, {\it $p$-adic analysis: a short course on recent work}, London Math. Soc. Lecture
Note Series, 46. Cambridge University Press, Cambridge-New York, 1980.

\bibitem{Liu}
C. Liu, {\it Twisted higher moments of Kloosterman sums}, Proc. Amer. Math. Soc.130 (2002), 1887--1892.

\bibitem{lin-tu}
Y. Lin and F. Tu, {\it Twisted Kloosterman sums}, J. Number Theory 147 (2015), 666--690.


\bibitem{Livne}
R. Livn\'{e}, {\it Motivic orthogonal two-dimensional representations of $\Gal(\overline{\Q}/\Q)$}, Israel J. Math. 92 (1995), 149-156.





%\bibitem{kob2}
%N. Koblitz, {\it The number of points on certain families of hypersurfaces over finite fields}, Compositio Math., 48 (1) (1983), 3--23.

%\bibitem {Lang} S. Lang, {\it Cyclotomic Fields I and II}, Graduate Texts in Mathematics, vol. 121, Springer-Verlag, New York, (1990).

%\bibitem{mccarthy1}
%D. McCarthy, {\it Extending Gaussian hypergeometric series to the $p$-adic setting}, Int. J. Number Theory 8 (7) (2012), 1581--1612.


\bibitem{mccarthy2}
D. McCarthy, {\it The trace of Frobenius of elliptic curves and the $p$-adic gamma function}, Pacific J. Math. 261 (2013), no. 1, 219--236.

%\bibitem{mccarthy3}
%D. McCarthy, {\it Transformations of well-poised hypergeometric functions over finite  fields}, Finite Fields and Their Applications, 18 (6) (2012), 1133--1147.

%\bibitem{mccarthy4}
%D. McCarthy, {\it On a supercongruence conjecture of Rodriguez-Villegas}, Proc. Amer. Math. Soc. 140 (7) (2012), 2241--2254.

%\bibitem{mortenson}
%Eric Mortenson, {\it Supercongruences for truncated ${_{n+1}F_n}$ hypergeometric series with applications to certain weight three newforms}, Proc. Amer. Math. Soc. 133 (2) (2005), 321--330.

\bibitem{mertens}
M. Mertens, {\it Mock modular forms and class numbers of quadratic forms}, Thesis (Ph.D.)-University of Cologne, 2014, 1--86.

\bibitem{mertens-2}
M. Mertens, {\it Mock modular forms and class number relations}, Res. Math. Sci. 1 (2014), Art. 6.

\bibitem{Moisio}
M. Moisio, {\it On the moments of Kloosterman sums and fibre products of Kloosterman curves}, Finite Fields Appl. 14 (2008) 515--531.



\bibitem{ono} K. Ono, \textit{Values of Gaussian hypergeometric series}, Trans. Amer. Math. Soc. 350 (1998), no. 3, 1205--1223.

\bibitem{KHN}
K. Ono, H. Saad and N. Saikia, {\it Distribution of values of Gaussian hypergeometric functions}, Pure Appl. Math. Q. 19 (2023), no. 1, 371--407.

\bibitem{PTM}
C. Peters, J. Top and M. van der Vlugt, {\it The Hasse zeta function of a $K3$ surface related to the number of words of weight 5 in the Melas codes}, J. Reine Angew. Math. 432 (1992), 151--176. 

\bibitem{Poincare}
H. Poincar\'{e}, {\it Fonctions modulaires et fonctions fuchsiennes}, Ann. Fac. Sci. Toulouse 3 (1911), 125--149.

\bibitem{Sa}
H. Sali\'{e}, {\it Zur Absch\"{a}tzung der Fourierkoeffizienten ganzer Modulformen} Math. Z. 36 (1933), 263--278.


\bibitem{Schmidt}
W. Schmidt, {\it Equations over Finite Fields: An Elementary Approach}, Lecture Notes in Math. 536, Springer, Berlin, 1976.

\bibitem{schoof}
 R. Schoof, {\it Nonsingular plane cubic curves over finite fields}, J. Comb. Theory Ser. A 46 (1987), no. 2, 183--211.

\bibitem{Silverman}
J. Silverman, {\it The arithmetic of elliptic curves}, Springer Verlag, New York, 1986.


\bibitem{Weil}
A. Weil, {\it On some exponential sums}, Proc. Nat. Acad. Sci. USA 34 (1948), 204--207.

\bibitem{Zagier}
D. Zagier, {\it Nombres de classes et formes modulaires de poids $3/2$}, C. R. Acad. Sci. Paris S\'{e}r. A-B 281 (1975), Ai, A883--A886.


\end{thebibliography}
\end{document}